\documentclass[11 pt]{amsart}

\usepackage{times}
\usepackage{geometry}
\usepackage{amssymb}
\usepackage{latexsym,amssymb,  amsmath, amscd, amsfonts}
\usepackage{graphicx}
\usepackage[percent]{overpic}
\usepackage{pdfsync}
\usepackage{units}
\usepackage{hyperref}
\usepackage{euscript}
\usepackage{multicol}
\usepackage{epstopdf}
\usepackage{paralist}

\newtheorem{theorem}{Theorem}

\newtheorem{proposition}[theorem]{Proposition}
\newtheorem{corollary}[theorem]{Corollary}

\theoremstyle{definition}
\newtheorem{definition}[theorem]{Definition}

\newtheorem{const}[theorem]{Construction}
\newtheorem{remark}[theorem]{Remark}

\newcommand{\R}{\mathbb{R}}      
\newcommand{\Z}{\mathbb{Z}}

\newcommand{\bdry}{\partial}

\def\Lk{{\operatorname{Lk}}}

\def\Rib{{\operatorname{Rib}}}

\def\Cr{{\operatorname{Cr}}}
\def\Len{{\operatorname{Len}}}

\newcommand{\ds}{\displaystyle}

\newcommand{\kwf}{K_{w,F}}
\newcommand{\kw}{K_{w}}

\newcommand{\lw}{L_{w}}
\def\pl{{\operatorname{pl}}}
\def\prs{{\operatorname{prs}}}

\begin{document}

\title{Ribbonlength of families of folded ribbon knots}
\author[Denne]{Elizabeth Denne}
\address{Elizabeth Denne: Washington \& Lee University, Department of Mathematics, Lexington VA}
\email[Corresponding author]{dennee@wlu.edu}
\urladdr{https://elizabethdenne.academic.wlu.edu/}
\author[Haden]{John Carr Haden}
\address{John Carr Haden: Washington \& Lee University}
\author[Larsen]{Troy Larsen}
\address{Troy Larsen: Washington \& Lee University}

\author[Meehan]{Emily Meehan}
\address{Emily Meehan: Smith College}
\curraddr{Gallaudet University, Science Technology \& Mathematics, Washington DC}

\date{\today}                                           
\makeatletter								
\@namedef{subjclassname@2020}{%
  \textup{2020} Mathematics Subject Classification}
\makeatother


\begin{abstract}
We study Kauffman's model of folded ribbon knots: knots made of a thin strip of paper folded flat in the plane. The folded ribbonlength is the length to width ratio of such a ribbon knot. We give upper bounds on the folded ribbonlength of 2-bridge, $(2,q)$ torus, twist, and pretzel knots, and these upper bounds turn out to be linear in  the crossing number. We give a new way to fold $(p,q)$ torus knots and show that their folded ribbonlength is bounded above by $2p$. This means, for example, that the trefoil knot can be constructed with a folded ribbonlength of 6. We then show that any $(p,q)$ torus knot $K$ with $p\geq q>2$ has a constant $c>0$, such that the folded ribbonlength is bounded above by $c\cdot\Cr(K)^{1/2}$. This provides an example of an upper bound on folded ribbonlength that is sub-linear in crossing number. 
 \end{abstract}

\maketitle

\section{Introduction}\label{sectionhistory}

Take a long thin strip of paper, tie a trefoil knot in it then gently tighten and flatten it. As can be seen in Figure~\ref{fig:trefoil-pentagon}, the boundary of the knot is in the shape of a pentagon. This observation is well known in recreational mathematics \cite{Ashley, John, Wel}. L. Kauffman \cite{Kauf05}  introduced a mathematical model of such a {\em folded ribbon knot}. Kauffman viewed the ribbon as a set of rays parallel to a polygonal knot diagram with the folds acting as mirrors, and the over-under information appropriately preserved. An overview of the history of folded ribbon knots can be found in \cite{Den-FRS}. 

\begin{center}
\begin{figure}[htbp]
\begin{overpic}{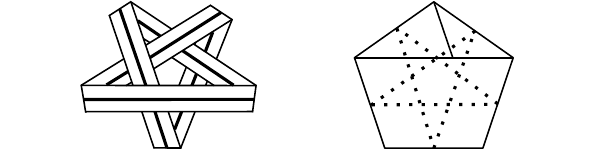}
\end{overpic}
\caption{On the left, a folded ribbon trefoil knot. On the right, the folded trefoil knot has been ``pulled tight'', minimizing the folded ribbonlength. Figure re-used with permission from \cite{Den-FRC}.}
\label{fig:trefoil-pentagon}
\end{figure}
\end{center}

One of the motivating questions about folded ribbon knots is the {\em folded ribbonlength problem}, which asks for the minimum folded ribbonlength $\Rib(K)$  for a particular knot (or link) type~$K$.  The {\em folded ribbonlength} of a folded ribbon knot  is the quotient of the length of the knot to the ribbon width. The ribbonlength problem can be viewed as a 2-dimensional version of the {\em ropelength problem} which asks for the minimum amount of rope needed to tie a knot in a rope of unit diameter.  (See for instance \cite{BS99,CKS,DEY,gm,lsdr}.) Indeed, many of the ideas in papers on ropelength have been used to make progress on questions about folded ribbonlength in this paper and others \cite{Den-FRS, Den-FRC,Den-FRLU}.

A separate, but equally interesting, open problem is to relate the minimum folded ribbonlength of a knot type $K$ to its crossing number $\Cr(K)$.  The {\em ribbonlength crossing number problem} asks for positive constants $c_1, c_2, \alpha, \beta$ such that
\begin{equation}
c_1\cdot \Cr(K)^\alpha \leq \Rib(K) \leq c_2\cdot \Cr(K)^\beta.
\label{eq:crossing}
\end{equation}
Y. Diao and R. Kusner \cite{DK} conjecture that $\alpha=\nicefrac{1}{2}$ and $\beta=1$ in Equation~\ref{eq:crossing}, and have work in progress showing $\alpha=\nicefrac{1}{2}$. 

In 2008, B. Kennedy, T.W. Mattman, R. Raya and Dan Tating \cite{KMRT} built on work of L. DeMaranville \cite{DeM} to make a first pass at this conjecture. They found linear upper bounds of folded ribbonlength in terms of crossing number for the $(q+1,q)$, $(2q+1,q)$, $(2q+2,q)$, and $(2q+4,q)$ families of torus knots. They also found quadratic upper bounds for $(2,q)$ torus knots.

In 2017, Grace Tian \cite{Tian} used grid diagrams to show that $\Rib(K)\leq 2\Cr(K)^2 + 6\Cr(K) +4$, thus showing $\beta\leq 2$ for all knots and links in the ribbonlength crossing number problem. In 2020, the first author \cite{Den-FRC}  improved Tian's result (by using arc presentations) to show that any non-split link type $L$ contains a folded ribbon link $\lw$ such that \begin{equation} 
\Rib(\lw)\leq \begin{cases} 0.32\Cr(L)^2 + 1.28\Cr(L) + 0.23, & \text{ when the arc index of $L$ is even.} \\
0.64\Cr(L)^2+2.55\Cr(L) + 2.03, & \text{ when the arc index of $L$ is odd.}
\end{cases} 
\label{eq:rib-squared}
\end{equation}

The first author \cite{Den-FRC} then used ideas in the ropelength results of \cite{DEY} to prove that any knot or link type $K$ contains a folded ribbon knot $\kw$ such that
\begin{equation}
\label{eq:bound}
\Rib(\kw)\leq 72\Cr(K)^{3/2} +32\Cr(K)+12\sqrt{\Cr(K)} +6.
\end{equation}
We note that the bound in Equation~\ref{eq:rib-squared} is smaller than Equation~\ref{eq:bound} when $\Cr(K)\leq 12,748$, so both results are useful. However, Equation~\ref{eq:bound} means that $\beta\leq\nicefrac{3}{2}$ for all knots  and links in the ribbonlength crossing number problem.

In Sections~\ref{sect:defn} and \ref{sect:useful-geom} we review the definition of a folded ribbon knot, and give a careful description of the geometry of a fold.  Then, in Proposition~\ref{prop:RibbonConstruct} and Corollary~\ref{cor:allowed-width} we use this geometry to show that for any polygonal knot diagram $K$  there is a constant $M>0$ such that a folded ribbon knot $\kw$ exists for all widths $w<M$. We also discuss the {\em  folded ribbonlength} of a folded ribbon knot. Formally, for a polygonal knot diagram $K$, the folded ribbonlength of the corresponding folded ribbon knot $\kw$ of width $w$ is $\Rib(\kw):=\nicefrac{\Len(K)}{w}$.

The rest of this paper is dedicated to finding upper bounds for folded ribbonlength for various families of knots and links.  In  Section~\ref{sect:2-bridge}, we use arguments from \cite{HHKNO} to  prove the following.
\begin{theorem}\label{thm:2-bridge}
Any 2-bridge knot type $K$ contains a folded ribbon knot $\kw$ such that 
$$\Rib(\kw) \le 6\Cr(K)-2.$$
\end{theorem}

In Section~\ref{sect:PTT}, we give a new method for folding $p$ half-twists. This method forms the base of a new construction for $(2,q)$ torus knots and the first construction for folded ribbon pretzel knots. These constructions then enable us to prove the following two theorems.
\begin{theorem}\label{thm:2p-torus}
Any  $(2,q)$ torus knot type $K$ contains a folded ribbon knot $\kw$ such that
$$\Rib(\kw) \le 2q=2\Cr(K).$$
\end{theorem}

\begin{theorem}\label{thm:pretzel}
Any  pretzel knot type $K=P(p,q,r)$ contains a folded ribbon knot $\kw$ such that 
 $$\Rib(\kw)\le2(|p|+|q|+|r|)+2.$$
\end{theorem}

In Section~\ref{sect:PTT}, we are reminded of two facts: (1) that for certain types of pretzel knots $P(p,q,r)$ we know the crossing number $\Cr(P(p,q,r))=|p|+|q|=|r|$, and (2) a twist knot is a specific type of pretzel knot. In Corollaries~\ref{cor:pretzel} and~\ref{cor:twist}, we use Theorem~\ref{thm:pretzel} to show that these two special families of knot types contain a folded ribbon knot $\kw$ such that 
$$\Rib(\kw)\leq 2\Cr(K) +2.$$
In particular, this means that the figure-eight knot can be constructed with a folded ribbonlength of 10. This is smaller than previously known results (\cite{Kauf05,Tian}).
Thus the 2-bridge, $(2,q)$ torus, twist, and certain pretzel knot families all give evidence that $\beta=1$ in the ribbonlength crossing number problem.

What can we say about $\alpha$ in the ribbonlength crossing number problem? There are a number of examples suggesting that $\alpha= \nicefrac{1}{2}$.  In \cite{Den-FRC}, the first author gives examples of families of link types $L$ where, for each $\nicefrac{1}{2}\leq p\leq 1$, the link type contains a folded ribbon link $\lw$ and a constant $c_p>0$ such that $\Rib(\lw)\leq c_p\cdot \Cr(K)^p$.  In Section~\ref{sect:pq-torus}, we give a new way to fold $(p,q)$ torus knots and links. The key is assuming that the torus knot has a greater number of turns around the longitude than the number of wraps around the meridian.  We use this new construction to prove the following.

\begin{theorem}\label{thm:pq-torus}
Any $(p,q)$ torus link type $L$ (with $p\geq q\ge2$) contains a folded ribbon link $\lw$ such that
 $$\Rib(\lw)\le 2p.$$
\end{theorem}

Theorem~\ref{thm:pq-torus} implies that the trefoil knot $T(3,2)$ can be constructed with a folded ribbonlength of 6. This is smaller than the previous known bounds (\cite{Kauf05,KMRT,Tian}). When we restrict our attention just to $T(p,2)$ torus knots, and use the fact that $T(p,2)$ and $T(2,p)$ torus knots are equivalent, then Theorem~\ref{thm:pq-torus} gives us a new proof of Theorem~\ref{thm:2p-torus}.
If we then look at torus knots with $p,q>2$, Theorem~\ref{thm:pq-torus} then enables us to prove the following.

\begin{theorem}\label{thm:torus-crossing}
Suppose $L$ is an infinite family of ($p,q$) torus link types, where $p, q> 2$ and $p=aq+b$ for some $a,b\in\Z_{\ge 0}$. Then for each $q=3,4,5,\dots$, the family $L$ contains a folded ribbon link $\lw$ with  
\begin{equation} \Rib(\lw)\le \sqrt{6}(a+\frac{b}{3})(\Cr(L))^\frac{1}{2}.
\label{eq:sublinear}
\end{equation}
\end{theorem}

For example, for each $q=3,4,5,\dots$, the $(q+1,q)$ torus link type contains a folded ribbon link $\lw$ with $\Rib(\lw)\le \frac{4}{3}\sqrt{6}\Cr(L)^{\frac{1}{2}}$. If we could prove that $\alpha = \nicefrac{1}{2}$ for the ribbonlength crossing number problem,  then Theorem~\ref{thm:torus-crossing} gives another set of examples showing that this is the best possible lower bound.  At the end of Section~\ref{sect:pq-torus}, we compare and contrast all the other known upper bounds on the ribbonlength of torus knots. It turns out that Theorem~\ref{thm:pq-torus} gives the lowest known bound. This is unsurprising since the folded ribbonlength bound  in Equation~\ref{eq:sublinear} is sub-linear in crossing number, while the previous work \cite{KMRT, Tian} gives linear and quadratic bounds.

We end this introduction by noting that our folded ribbons are centered about polygonal knot diagrams. We will describe this in great detail in Section~\ref{sect:defn}, but note here that polygonal knots are made up of a finite number of straight edges or sticks.  For these kinds of knots, it is natural to think about the stick number of the knot (or link) type. The stick number $s(K)$ is the least number of straight sticks, joined end to end, needed to create the knot $K$ (see \cite{Adams}).  The reader might think that there is a close connection between the minimum folded ribbonlength and stick number.  However, in Section~\ref{sect:conclusion}, we show that this is (surprisingly) not the case. The relationship is far more subtle and is a source of yet more open questions about folded ribbon knots.

\section{Folded ribbon knots, ribbon equivalence, and ribbonlength}\label{sect:defn}

\subsection{Modeling folded ribbon knots}\label{sect:model}
Following knot theory texts \cite{Adams,Crom, Liv}, we define a {\em knot} $K$ to be a simple closed curve in $\R^3$ that is equivalent to a polygonal knot.  Two knots are equivalent if one can be moved to the other in space without a strand of a knot passing through itself. (They are ambient isotopic.) A {\em link} is a finite disjoint union of knots. While most of the definitions and results in this paper are stated for knots, they also hold for links.  A {\em polygonal knot} $K$ has a finite number of vertices $v_1,\dots, v_n$ and edges $e_1=[v_1,v_2], \dots , v_n=[v_n,v_1]$.  If $K$ is oriented, we assume that the numbering of the vertices follows the orientation. Sometimes the edges of a polygonal knot are called {\em sticks}.

A {\em polygonal knot diagram} is a projection of a polygonal knot to a plane, where the crossing information of overlapping strands has been preserved. This is usually done at a crossing by adding a break in the strand which lies below another. It is important to note that when modeling folded ribbon knots, we do not have to assume that our knot diagrams are regular\footnote{A  polygonal knot diagram is {\em regular} (see \cite{Liv}) provided no three points on the knot project to the same point, and no vertex projects to the same point as any other point on the knot.}.  For example, take a strip of paper and join to create an annulus. Now flatten, creating two folds. This is a folded ribbon unknot whose knot diagram has two overlapping edges, where we remember that one edge is over the other. 

Formally, a folded ribbon knot is a piecewise linear immersion of an annulus or M\"obius band into the plane, where the fold lines are the only singularities, and where the crossing information is consistent. The example of the folded ribbon unknot with two edges fits in here because immersions are locally one-to-one.  (Note that a technical discussion of polygonal knot diagrams and the definition of a folded ribbon knot can be found in \cite{DKTZ}, and a general discussion for smooth ribbon knots immersed in the plane can be found in \cite{DSW}.)  We observe that the folded ribbon knot has two kinds of double points in the plane. The first are those near crossings of the knot diagram,  and the second are those near the fold lines of the ribbon. Specifically, we define a {\em fold} to be a connected component of the set of double points which lifts to a single component containing the fold line in the preimage. For example,  on the left in Figure~\ref{fig:ribbon-construct}, the fold is the set of double points found in triangle $\Delta ABD$.

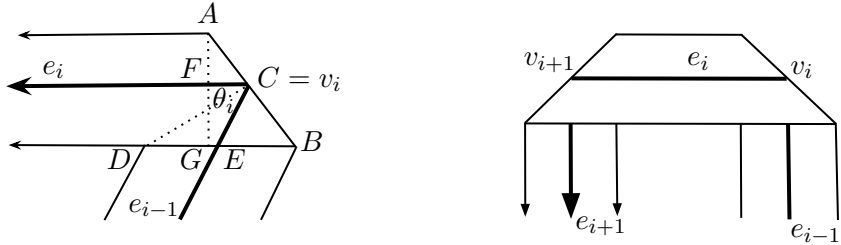
\begin{figure}[htbp]
\begin{center}
\begin{overpic}{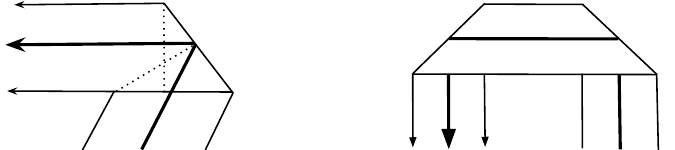}
\put(23,23){$A$}
\put(35,8){$B$}
\put(30,15.5){$C=v_i$}
\put(24.8,13){$\theta_i$}
\put(26,6){$E$}
\put(12.5,6){$D$}
\put(21,16.5){$F$}
\put(21,6){$G$}
\put(5,17){$e_i$}
\put(15,1){$e_{i-1}$}

\put(80,18){$e_i$}
\put(92,-2){$e_{i-1}$}
\put(67,-1){$e_{i+1}$}
\put(92,17){$v_i$}
\put(61,18){$v_{i+1}$}
\end{overpic}
\caption{On the left, a close-up view of a ribbon fold, with fold angle $\theta_i=\angle FCE$. On the right, the construction of the ribbon centered on edge $e_i$. Figure re-used with permission from~\cite{DKTZ}. }
\label{fig:ribbon-construct}
\end{center}
\end{figure}

Practically, how do we construct a folded ribbon knot?  The key idea is to start with a polygonal knot diagram $K$ and build a folded ribbon of width $w$ with $K$ as the centerline of folded ribbon.  To describe our construction, we first define the {\em fold angle} at vertex $v_i$ to be the angle $\theta_i$ (where $0\leq\theta_i\leq \pi$) between edges $e_{i-1}$ and $e_i$. (This angle is said to be positive when $e_i$ is to the left of $e_{i-1}$ as in Figure~\ref{fig:ribbon-construct} (left), and is negative when $e_i$ is to the right.)  Second, we remember that the boundary of the ribbon is distance $\nicefrac{w}{2}$ from $K$, and because the fold lines act as mirrors to rays parallel to the knot diagram, we know the fold line is perpendicular to the angle bisector of~$\theta_i$. In Figure~\ref{fig:ribbon-construct} (left), the fold angle is $\angle ECF=\theta_i$, the fold line $AB$ is perpendicular to the angle bisector $DC$ of $\theta_i$.  Finally, given that the width of the ribbon $w=|AG|$, we use $\triangle AGB$ to determine the length of the fold line $|AB|=\frac{w}{\cos(\nicefrac{\theta_i}{2})}$. As we see next in Construction~\ref{const:folded-ribbon} we use the fold lines to start our folded ribbon construction.

\begin{const} [\cite{DKTZ} Definition 2.5] \label{const:folded-ribbon}
Given an oriented polygonal knot diagram $K$, we construct a folded ribbon knot $K_w$ of width $w$ as follows.
\begin{compactenum}
\item If the fold angle $\theta_i<\pi$, we place a fold line of length $\frac{w}{\cos(\nicefrac{\theta_i}{2})}$ centered at $v_i$ perpendicular to the angle bisector of $\theta_i$.   If $\theta_i=\pi$, there is no fold line.
\item We add in the ribbon boundaries by joining the ends of the fold lines at $v_i$ and $v_{i+1}$. 
\item The folded ribbon knot inherits an orientation from $K$.
\end{compactenum}
\end{const}

At each fold, there is a choice of which piece of the ribbon lies over the other. This choice gives the folding information $F$ of a folded ribbon knot. 
\begin{definition}[\cite{DKTZ} Definition 2.6]
Let $K_w$ be an oriented folded ribbon knot which is immersed (with singularities at the fold lines). There is an {\em overfold} at vertex $v_i$ if the piece of ribbon corresponding to segment $e_{i}$ is over the piece of ribbon corresponding to segment $e_{i-1}$ (see Figure~\ref{fig:folding-info} right). Similarly, there is an {\em underfold} if the piece of ribbon corresponding to $e_{i}$ is under the piece of ribbon corresponding to $e_{i-1}$. The choice of over or underfold at each vertex of $K_w$ is called the {\em folding information}, and is denoted by $F$.
\end{definition}

\begin{figure}[htbp]
\begin{center}
\begin{overpic}{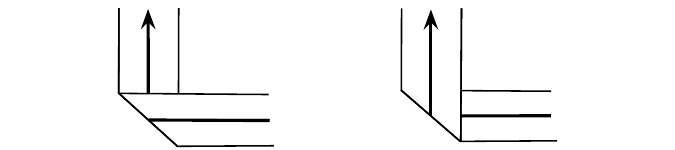}
\put(36, 6){$e_{i-1}$}
\put(23, 20){$e_{i}$}
\put(17, 4){$v_i$}

\put(76, 6.5){$e_{i-1}$}
\put(65, 20){$e_{i}$}
\put(60, 4){$v_i$}
\end{overpic}
\caption{A right underfold (left) and a right overfold (right). Figure re-used with permission from~\cite{DKTZ}.}
\label{fig:folding-info}
\end{center}
\end{figure}
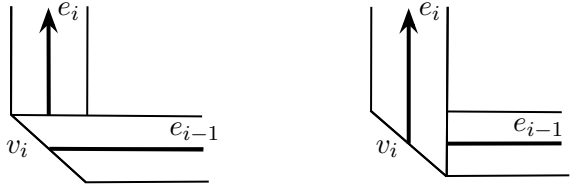

Remembering that the folded ribbon is an immersion of an annulus or M\"obius band into the plane, we observe that the width of the ribbon cannot be too large, or the ribbon will have self-intersections. This can be  imagined from Figure~\ref{fig:ribbon-construct} right, where the fold lines at vertices $v_i$ and $v_{i+1}$ will eventually meet if $w$ increases enough. More formally, we have the following.

\begin{definition}[\cite{DKTZ} Definition 2.7] \label{def:allowed}
Given an oriented polygonal knot diagram $K$, we say the folded ribbon knot (denoted by $\kwf$ or $\kw$), of width $w$ and folding information $F$, is {\em allowed} provided
\begin{enumerate}
\item The ribbon has no singularities (is immersed), except at the fold lines which are assumed to be disjoint.
\item $\kwf$ has consistent crossing information, and moreover this agrees
\begin{enumerate}
\item with the folding information given by $F$, and
\item with the crossing information from the knot diagram $K$.
\end{enumerate}
\end{enumerate}
\end{definition}

Throughout the rest of this paper, {\bf we will assume that all of the folded ribbons $\kwf$ or $\kw$ are allowed.} (This is not an unreasonable assumption. We prove in Section~\ref{sect:RibbonConstruct} Corollary~\ref{cor:allowed-width} and Remark~\ref{rmk:allowed-width} that a knot diagram has an allowed folded ribbon knot for small enough widths.) We will also frequently use the notation $\kw$ for an allowed ribbon where we do not need to keep track of the folding information $F$, as will be discussed in the next two sections.

Before we move on, we pause to give a more detailed description of a fold which will be very useful both in Section~\ref{sect:useful-geom} and later on.  In Figure~\ref{fig:AcuteVsObtuse1}, we see two folds (the double points in $\Delta ABF$)  and we note the exact geometry of them depends on whether the fold angle $\theta$ is acute or obtuse.

\begin{center}
\begin{figure}[htbp]
\begin{overpic}{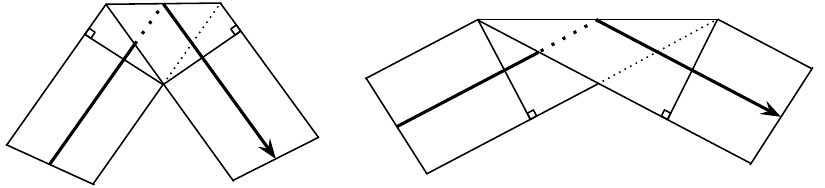}
\put(12,23){$A$}
\put(19,23){$C$}
\put(19.15, 19){$\theta$}
\put(26,23){$B$}
\put(17,16.5){$D$}
\put(11.5,14){$E$}
\put(18.5,9){$F$}
\put(8,18.5){$G$}
\put(20,16.5){$H$}
\put(23.5, 13){$I$}
\put(57.5,21){$A$}
\put(72,21){$C$}
\put(71.5, 17.5){$\theta$}
\put(86,21){$B$}
\put(67.5,15.5){$D$}
\put(57.5,14){$E$}
\put(71.5,9.5){$F$}
\put(64,6.25){$G$}
\put(78,13.5){$H$}
\put(83.4,11){$I$}
\end{overpic}
\caption{The fold from angle $\theta$ consists of two overlapping triangles of ribbon, which correspond to the part of the knot diagram $DC+CH$. Similarly, the  extended fold from angle $\theta$ is the two pieces of ribbon determined by $EC+CI$. Figure re-used with permission from \cite{Den-FRC}.}
\label{fig:AcuteVsObtuse1}
\end{figure}
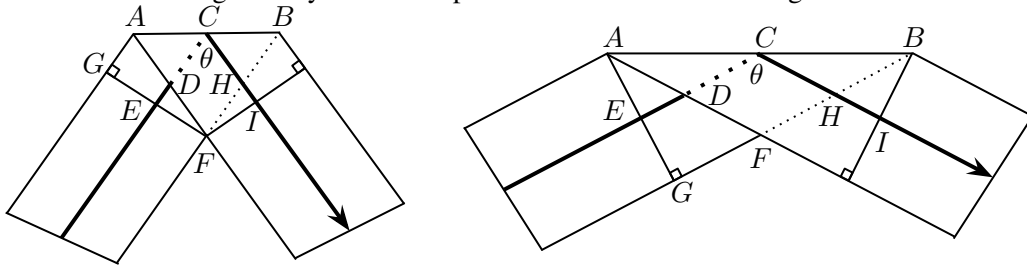
\end{center}

We have oriented the knot diagram and ribbon in Figure~\ref{fig:AcuteVsObtuse1} so that the fold $\Delta ABF$ is an overfold in both cases. Since we use the knot diagram to construct our folded ribbon knot, we can talk about the pieces of the ribbon which correspond to particular parts of the knot diagram.  In Figure~\ref{fig:AcuteVsObtuse1}, the fold consists of two identical triangles of ribbon. One triangle of ribbon corresponds to the part of the knot diagram $DC$, and this triangle lies beneath the triangle of ribbon corresponding to the part of the knot diagram $CH$.  The easiest way to see this clearly is to take a strip of paper and make a fold!  Eventually, we will want to talk about the  geometry of folds corresponding to a particular fold angle $\theta$.  We combine all of this discussion in the following definition.

\begin{definition} Assume that the fold angle $\theta$ satisfies $0<\theta < \pi$. Using the notation in Figure~\ref{fig:AcuteVsObtuse1}, we define the {\em fold from angle $\theta$} is be the two identical overlapping triangles of ribbon, determined by the knot diagram $DC+CH$. The {\em extended fold from angle $\theta$} is defined to to be the two congruent pieces of ribbon determined by  $EC+CI$. 
\label{def:extended-fold}
\end{definition}

\begin{remark} Why the restriction on the fold angle $\theta$? As we saw in Construction~\ref{const:folded-ribbon}, when $\theta=\pi$, there is no fold. The other extreme case is when $\theta=0$, which we saw previously in the unknot whose diagram is made from two overlapping edges. In the $\theta=0$ case, the fold consists of two overlapping rectangles or trapezoids of ribbon. Exactly which shape depends on what the knot diagram does in the edges adjacent to the two edges creating the fold from angle $0$. In this case, it does not make sense to talk about an extended fold.

We also note the special case when the fold angle $\theta =\nicefrac{\pi}{2}$.  This angle is the transition point between the acute and obtuse folds shown in Figure~\ref{fig:AcuteVsObtuse1}.   In this case, the fold and the extended fold are identical. 
\label{rmk:extended-fold}
\end{remark}

\subsection{Folded ribbon equivalence}
Given our understanding of a folded ribbon knot it is natural to wonder when two folded ribbon knots are equivalent. We give a summary of the ideas here, more details can be found in \cite{Den-FRS, DKTZ}.  The easiest way to think about folded ribbon knot equivalence is to ignore the folded ribbon and study the knot diagram. 

\begin{definition}[\cite{DKTZ} Definition 3.4] Two folded ribbon knots are {\em diagram equivalent} if they have equivalent knot diagrams.
\end{definition}

Another way to understand folded ribbon equivalence is to keep track of the folded ribbon and folding information. This can be done by keeping track of the knot diagram, the topological type of the ribbon, and the way the ribbon twists about the knot diagram. Since an oriented folded ribbon knot is a framed knot, we can can compute the linking number between the knot diagram and one boundary component.  Recall that the {\em linking number} is a link invariant used to determine the degree to which components of a link are joined together.  Given an oriented two component link $L=A\cup B$, the linking number  $\Lk(A,B)$ is defined to be one half the sum of $+1$ crossings and $-1$ crossings between $A$ and $B$. (See Figure~\ref{fig:crossingvalue}.)

\begin{figure}[htbp]
\begin{center}
\begin{overpic}{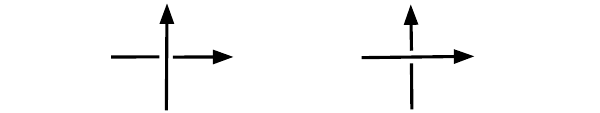}
\end{overpic}
\end{center}
\caption{The crossing on the left is labelled -1, the crossing on the right +1. Figure re-used with permission from~\cite{DKTZ}.}
\label{fig:crossingvalue}
\end{figure}

 \begin{definition} [\cite{DKTZ} Definition 3.1] \label{ribbonlinkingnumber}
 Given an oriented folded ribbon knot $K_{w,F}$, we define the {\em ribbon linking number} to be the linking number between the knot diagram and one boundary component of the ribbon.  We denote this as $\Lk(K_{w,F})$, or $\Lk(K_w)$.
 \end{definition}

\begin{definition} Two oriented folded ribbon knots (or links) are {\em (folded) ribbon equivalent} if they are diagram equivalent, have the same ribbon linking number, and are both topologically equivalent either to a M\"obius band or an annulus when considered as ribbons in $\R^3$.
\label{def:ribbon-equivalence}
\end{definition}

For example, the left and center folded ribbon unknots in Figure~\ref{fig:4unknot} are ribbon equivalent (with ribbon linking number 0), while the one on the right is not ribbon equivalent to them (with ribbon linking number $-2$). This example shows that there can be different looking folded ribbon knots with the same ribbon linking number. 

The topological type of the ribbon is also needed in Definition~\ref{def:ribbon-equivalence}. For example, a convex 4-stick folded ribbon unknot and a 3-stick folded ribbon unknot can each have ribbon linking number $\pm 1$. For these folded ribbon unknots, we simply need one fold to be of different type from the others. (Note that the folded ribbon knot is a topological annulus when the number of edges in the knot diagram is even, and is a M\"obius band when the number of edges is odd.)
\begin{figure}[htbp]
\begin{center}
\begin{overpic}{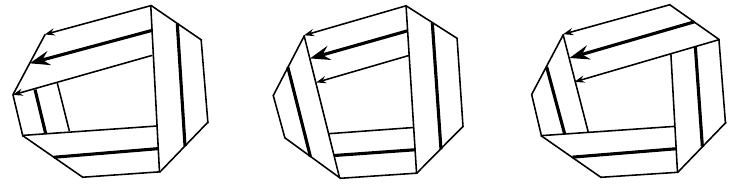}
\end{overpic}
\begin{overpic}{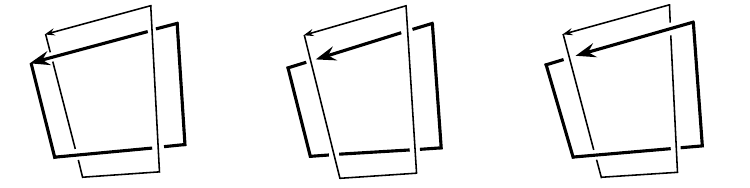}
\end{overpic}
 \caption{Three different 4-stick folded ribbon unknots (above) with corresponding link diagrams of one boundary component and the knot diagram (below). The left and center ribbon unknots have ribbon linking number 0, while the one on the right has ribbon linking number $-2$. Figure re-used with permission from \cite{Den-FRS}.} 
\label{fig:4unknot}
\end{center}
\end{figure}

\subsection{Ribbonlength}\label{sect:ribbonlength}

As discussed in the introduction, we can ask to find the least length of paper needed to tie a folded ribbon knot. To begin to answer this question, we define a scale-invariant quantity called {\em folded ribbonlength}.
\begin{definition}[\cite{Kauf05}] The {\em folded ribbonlength} of a folded ribbon knot $\kwf$  of width $w$  and folding information $F$ is the quotient of the length of $K$ to the width $w$:
$$ \Rib(\kwf) = \Rib(\kw)= \frac{\Len(K)}{w}.$$
\end{definition}

The {\em (folded) ribbonlength problem} asks us to ``minimize'' the folded ribbonlength of a folded ribbon knot. Given our definition of an allowed folded ribbon, this means that for a particular knot type, we can fix the width and attempt to find the infimum of length. For example,  consider the folded ribbon unknot (described above) with knot diagram consisting of just two edges. The ribbonlength here is zero, since we can fix width $w=1$ and find $\Len(K)\rightarrow 0$.

We alert the reader that this version of the folded ribbonlength problem asks to minimize ribbonlength with respect to diagram equivalence of folded ribbon knots. Since  this case allows us to ignore the folding information and ribbon linking number, {\em we usually use the notation $\kw$ for our folded ribbon knots and $\Rib(\kw)$ for the folded ribbonlength}.  A way into this problem is to set the width $w=1$ and find a length of a knot $K$ diagram of a folded ribbon knot $\kw$. This computation then gives an upper bound on the minimum folded ribbonlength for the knot type of $K$. The majority of the papers written to date fall into this context. (See for instance \cite{DeM, Den-FRC, Den-FRS, Kauf05, KMRT, Tian}.)  Indeed, the rest of this paper is written in this context. We will not attempt to keep a track of the ribbon linking number of the various families of folded ribbon knots we construct.

In previous work \cite{DKTZ}, we have made a first pass at putting the ribbonlength problem in the context of ribbon equivalence. That is, keeping a track of both the ribbon linking number and topological type of the ribbon. Amongst other results, we show that a 3-stick unknot with linking number $\pm 3$ has minimum ribbonlength $3\sqrt{3}$ when the unknot diagram is an equilateral triangle. We also show that the ribbonlength of a 3-stick unknot with ribbon linking number $\pm 1$ is bounded above by $\sqrt{3}$.  In work-in-progress \cite{Den-FRLU}, we continue to examine unknots, giving bounds on ribbonlength that depend on the linking number of the unknot.

\section{Local geometry of folded ribbons}\label{sect:useful-geom}

In this section we prove several useful results about the local geometry of a fold in a folded ribbon knot. We will soon use these results to prove that for any regular polygonal knot diagram $K$, there is a corresponding allowed folded ribbon knot $\kw$, provided the width $w$ is small enough.

\begin{proposition} \label{prop:basic-geom}
Let $\kw$ be a folded ribbon knot, and assume that the ribbon near a fold line is labeled as in Figure~\ref{fig:basic-geometry} (left). Assume that the fold angle $\angle DCH=\theta$, where $0<\theta<\pi$. Then
\begin{enumerate}
\item $\angle DAC= \angle ACD = \angle HCB = \angle CBH = \nicefrac{\pi}{2} -\nicefrac{\theta}{2}$,
\item $|AD|=|DC| = |CH| = |HB| = |DF| = |FH|$.
\item $CDFH$ is a rhombus,
\item $|FC| = \frac{w}{2\sin\nicefrac{\theta}{2}}$,
\item $|AD|= \frac{w}{2\sin\theta}$.
\end{enumerate}
\end{proposition}

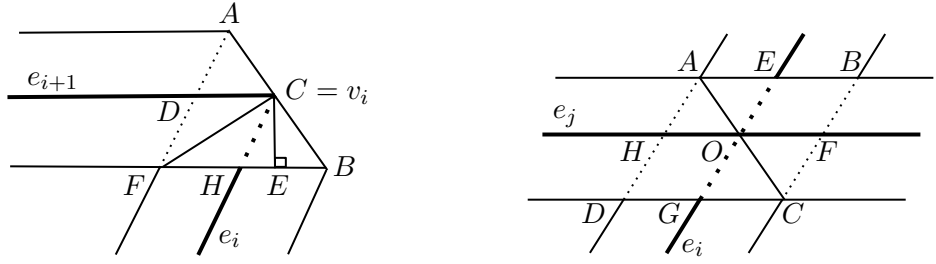
\begin{figure}[htbp]
\begin{center}
\begin{overpic}{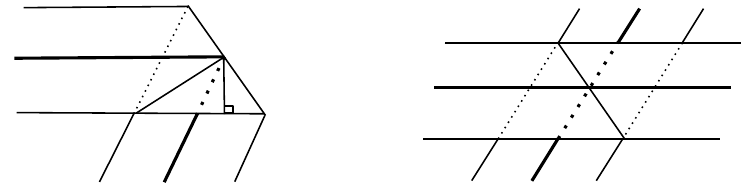}
\put(24,25){$A$}
\put(36,9){$B$}
\put(31,17){$C=v_i$}
\put(17.5,14.75){$D$}
\put(29,7){$E$}
\put(14,7){$F$}
\put(22,7){$H$}
\put(4,18.5){$e_{i+1}$}
\put(24,2){$e_{i}$}
\put(59,15){$e_j$}
\put(72.5,1){$e_i$}
\put(72,20){$A$}
\put(89,20){$B$}
\put(83,4){$C$}
\put(62,4){$D$}
\put(80,20){$E$}
\put(86.5,10.5){$F$}
\put(70,4){$G$}
\put(66,10.5){$H$}
\put(74.5,10.5){$O$}
\end{overpic}
\end{center}
\caption{Close-up of the overlapping ribbon near a fold (left) and at a crossing (right).}
\label{fig:basic-geometry}
\end{figure}

\begin{proof}
(1) By construction $FC$ is the bisector of $\angle DCH$ and is perpendicular to $AB$. Hence $\angle DCA=\angle HCB = \nicefrac{\pi}{2} -\nicefrac{\theta}{2}$. Now use the fact that corresponding angles between parallel lines are equal to deduce that $\angle DAC=\angle HCB$ and $\angle ACD=\angle CBH$.

(2) By construction $|AC|=|CB|$, and by Angle-Side-Angle we deduce  $\triangle ACD \cong \triangle BCH$. These two triangles are also isosceles by (1), and hence $|AD|= |DC| = |CH| = |HB|$.
We observe that $\angle FAB=\angle ABF= \nicefrac{\pi}{2} -\nicefrac{\theta}{2}$, hence triangle $\triangle AFB$ is similar to both $\triangle ADC$ and $\triangle BCH$. Since $|AC| = \nicefrac{|AB|}{2}$, similarity allows us to deduce $\nicefrac{|AF|}{2}=|AD|=|DF|$, and $\nicefrac{|FB|}{2}=|HB|=|FH|$.  

(3) That $CDFH$ is a rhombus follows from (2).

(4) Observe that $\angle CFB=\nicefrac{\theta}{2}$ since $FC$ is a diagonal of rhombus $CDFH$. Drop a perpendicular $CE$ from $C$ to $FB$. Then $|CE| = \nicefrac{w}{2}$, and we use right-triangle $\triangle FEC$ to deduce $\ds |FC|=\frac{|CE|}{\sin\nicefrac{\theta}{2}} = \frac{w}{2\sin\nicefrac{\theta}{2}}$. 

(5)  Since $\angle CHE$ is exterior to $\triangle FCH$, and $FC$ is a diagonal in $CDFH$, we deduce that $\angle CHE=\theta$.  Using the same perpendicular $CE$ from (4) and right-triangle $\triangle CEH$, we deduce 
$\ds|CH|=\frac{|CE|}{\sin\theta}=\frac{w}{2\sin\theta}$. Now combine this with (2).
\end{proof}

\begin{corollary} \label{cor:basic-geom}
Let $\kw$ be a folded ribbon knot. Assume that edge $e_i$  of $K$ crosses edge $e_j$ transversally at point $O$ as shown in Figure~\ref{fig:basic-geometry} right.  Using the notation in Figure~\ref{fig:basic-geometry} right, let $\angle GOH=\theta$. Then
\begin{compactenum}
\item the overlapping ribbon forms a rhombus $ABCD$,  
\item the diagonals of $ABCD$ have length $|AC|= \frac{w}{\cos\nicefrac{\theta}{2}}$, and $|BD| =  \frac{w}{\sin\nicefrac{\theta}{2}}$.
\end{compactenum}
\end{corollary}

\begin{proof} (1) follows from the geometry of parallel lines. For (2), note that if we join $AC$, we can imagine this forms a fold as in Figure~\ref{fig:basic-geometry} (left). This means $|AC|$ corresponds to the length of the fold (found in Construction~\ref{const:folded-ribbon}). Next, note that $|BD|=2|OD|$, and $OD$ corresponds to $FC$ from Proposition~\ref{prop:basic-geom}.
\end{proof}

We continue our examination of the local geometry of a fold as shown in Figure~\ref{fig:AcuteVsObtuse}. Recall from Definition~\ref{def:extended-fold} and Remark~\ref{rmk:extended-fold} that the most relevant case is where the fold angle satisfies $0<\theta<\pi$. Then {\em fold from angle $\theta$} consists of two identical overlapping triangles of ribbon, determined by the knot diagram $DC+CH$, and the {\em extended fold from angle $\theta$} is the two congruent pieces of ribbon determined by  $EC+CI$. 

\begin{center}
\begin{figure}[htbp]
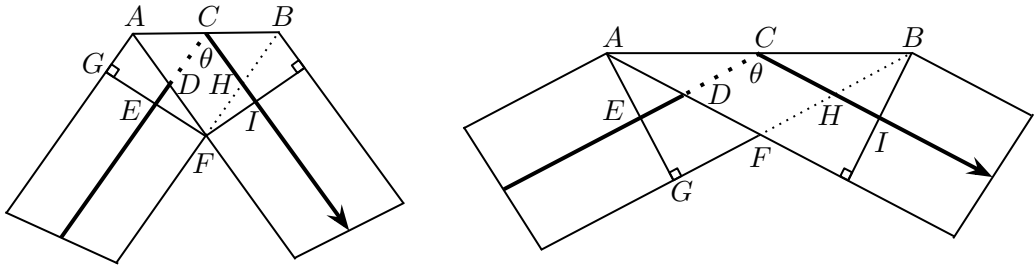

\begin{overpic}{AcuteVsObtuse}
\put(12,23){$A$}
\put(19,23){$C$}
\put(26,23){$B$}
\put(19.15, 19){$\theta$}
\put(17,16.5){$D$}
\put(11.5,14){$E$}
\put(18.5,9){$F$}
\put(8,18.5){$G$}
\put(20,16.5){$H$}
\put(23.5, 13){$I$}
\put(57.5,21){$A$}
\put(72,21){$C$}
\put(71.5, 17.5){$\theta$}
\put(86,21){$B$}
\put(67.5,15.5){$D$}
\put(57.5,14){$E$}
\put(71.5,9.5){$F$}
\put(64,6.25){$G$}
\put(78,13.5){$H$}
\put(83.4,11){$I$}
\end{overpic}
\caption{On the left a fold where the fold angle $\angle DCH=\theta$ is acute, on the right a fold where the fold angle is obtuse. Figure re-used with permission from \cite{Den-FRC}.}
\label{fig:AcuteVsObtuse}
\end{figure}
\end{center}

\begin{proposition} \label{prop:fold-lengths}
Take a folded ribbon knot of width $w$.  Referring to the notation in Figure~\ref{fig:AcuteVsObtuse}, assume the fold angle $\angle DCH=\theta$, where $0< \theta<\pi$. Then the length of the knot diagram
\begin{compactenum} 
\item  in a fold from angle $\theta$ is 
$$|DC|+|CH|= \frac{w}{\sin\theta};$$
\item in an extended fold from angle $\theta$  is 
$$|EC|+|CI|=\begin{cases} w\cot(\nicefrac{\theta}{2})
 & \text{ when $0<\theta\leq \nicefrac{\pi}{2}$,} \\
w\cot(\nicefrac{\pi}{2}-\nicefrac{\theta}{2})& \text{when $\nicefrac{\pi}{2}\leq \theta< \pi$.}
\end{cases}
$$
\end{compactenum}
\end{proposition}

\begin{proof} For acute fold angles, we drop a perpendicular $FG$ from $F$ to $AG$ as in Figure~\ref{fig:AcuteVsObtuse} (left). Then $|GF|=w$ and $|EF|=\nicefrac{w}{2}$. Figure~\ref{fig:AcuteVsObtuse} (right) shows a similar construction for obtuse angles.

Case 1: Assume the fold angle $\theta$ is acute, $0<\theta\leq \nicefrac{\pi}{2}$, as in Figure~\ref{fig:AcuteVsObtuse} (left). Using properties of parallel lines, we see  $\angle DCH=\angle GAF=\angle EDF= \theta.$
In right-triangle $\triangle AGF$, we find $|AF|=\ds\nicefrac{w}{\sin\theta}$. We also know that $|DC|=|CH|=|AD|=|DF|$ from Proposition~\ref{prop:basic-geom}, and therefore $|DC|+|CH|=|AF|=\nicefrac{w}{\sin\theta}$.

Now use right-triangle $\triangle DEF$ to find  $|DE|=(\nicefrac{w}{2})\cot\theta$. To find the length in (2), we use standard trig identities. Then the length of the the knot diagram of an extended fold is 
$$|EC|+|CI|=\frac{w}{\sin\theta}+w\cot\theta = \frac{w(1+\cos\theta)}{\sin\theta}=w\cot(\nicefrac{\theta}{2}).
$$

 Case 2: Assume that the fold angle $\theta$ is obtuse, $ \nicefrac{\pi}{2}\leq \theta< \pi$, as in Figure~\ref{fig:AcuteVsObtuse} (right). From Proposition~\ref{prop:basic-geom}, we know $CDFH$ is a rhombus. This means that angle $\angle DFH=\theta$, and hence angle $\angle AFG=\angle ADE=\pi-\theta$. Then, using right-triangle $\triangle AGF$ we find $|AF|=\ds\nicefrac{w}{\sin(\pi-\theta)}=\nicefrac{w}{\sin\theta}$. Next use right-triangle $\triangle AED$, to find $|DE|=\ds(\nicefrac{w}{2})\cot(\pi-\theta)$. Combining these two results with standard trig identities, we find
$$|EC|+|CI|= \frac{w(1+\cos(\pi-\theta))}{\sin(\pi-\theta)}=w\cot(\nicefrac{\pi}{2}-\nicefrac{\theta}{2}).
$$
\end{proof}

\begin{corollary} \label{cor:extended-fold}
The ribbonlength
\begin{compactenum} 
\item  of a fold from angle $\theta$ is 
$$\Rib(|DC|+|CH|)= \frac{1}{\sin\theta};$$
\item of an extended fold from angle $\theta$ is 
$$\Rib(|EC|+|CI|)=\begin{cases} \cot(\nicefrac{\theta}{2})
 & \text{ when $0<\theta\leq \nicefrac{\pi}{2}$,} \\
\cot(\nicefrac{\pi}{2}-\nicefrac{\theta}{2})& \text{when $\nicefrac{\pi}{2}\leq \theta< \pi$.}
\end{cases}
$$
\end{compactenum}
\end{corollary}
\begin{proof} Since the ribbonlength is defined to be the quotient of the length to the thickness, simply divide the lengths in Proposition~\ref{prop:fold-lengths} by $w$.
\end{proof}

\begin{corollary}\label{cor:triangles}
The minimum folded ribbonlength needed to create a fold is 1, and occurs when $\theta=\nicefrac{\pi}{2}$. In this case the fold consists of two identical right isosceles triangles whose equal sides have length $w$.
\end{corollary}
\begin{proof}  From Corollary~\ref{cor:extended-fold} we have the ribbonlength of a fold from angle $\theta$ is 

$$\Rib(|DC|+|CH|)= \frac{1}{\sin\theta}, \quad \text{for $0< \theta<\pi$.}$$ 
This function has a minimum of 1 when  $\theta=\nicefrac{\pi}{2}$. The fold is, by construction, two identical right isosceles triangles whose equal sides have length $w$.
\end{proof}

\begin{remark} We saw in Remark~\ref{rmk:extended-fold} that the fold angle $\theta=\nicefrac{\pi}{2}$ is the transition point between the acute and obtuse folds shown in Figure~\ref{fig:AcuteVsObtuse}. When $\theta=\nicefrac{\pi}{2}$,  the fold and the extended fold are identical. Thus Corollary~\ref{cor:triangles} also gives the minimum folded ribbonlength of an extended fold.
\end{remark}

\subsection{Constructing allowed folded ribbon knots}\label{sect:RibbonConstruct}

We will use the ideas given above to show that allowed folded ribbon knots are possible, provided with width $w$ is small enough. 

\begin{proposition}\label{prop:RibbonConstruct} Given any regular polygonal knot diagram $K$ and folding information $F$, there is a constant $M>0$ such that for all widths $w<M$, the corresponding folded ribbon knot $\kw$ has only single and double points. 
\end{proposition}

\begin{proof}
Since $K$ is a regular, polygonal knot diagram, it has only single and double points (and can be viewed as the $w=0$ case). Let $\mathcal {W}\subset \R_{\geq 0}$ be the set of non-negative widths for which the corresponding folded ribbon knot $\kw$ has only single and double points.  We know that $\mathcal{W}$ is non-empty since $0\in \mathcal{W}$. We will show that $\mathcal{W}$ is open in $\R_{\geq 0}$. To do this, we use the standard (subspace) topology on $\R$. We show that for a nonzero width $w_0\in\mathcal{W}$, there is an $\varepsilon>0$, such that $(w_0-\varepsilon, w_0+\varepsilon)\subset \mathcal{W}$. Then $\mathcal{W}$ open implies there is a constant $M>0$,  such that $[0,M)\subset\mathcal{W}$.

 Pick a width $w_0\in \mathcal{W}$. We will now show that we can increase the width and still stay in $\mathcal{W}$.  To do this, we think about the structure of the folded ribbon around the knot diagram $K$.   From Proposition~\ref{prop:basic-geom}, and Figure~\ref{fig:basic-geom-const} (left), we know there is one rhombus (region 1)  and two triangles (regions 2 and 3) of double-points of ribbon at a fold of $\kw$. From  Corollary~\ref{cor:basic-geom}, and Figure~\ref{fig:basic-geom-const} (right), we know there are four rhombi (regions 1 through 4) of double-points of ribbon at a crossing of $\kw$. 

\begin{center}
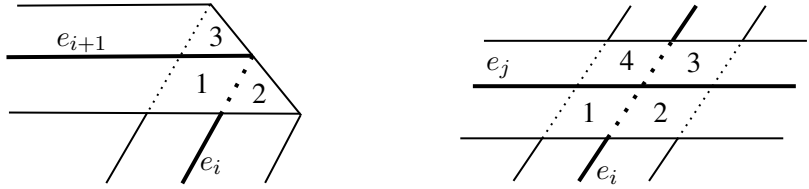
\begin{figure}[htbp]
\begin{overpic}{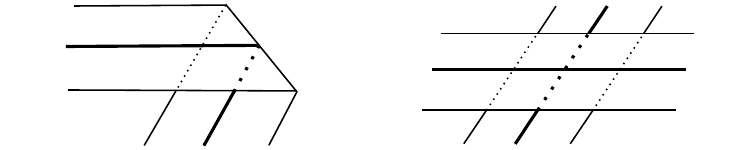}
\put(28.5, 10){1}
\put(34.5,9){2}
\put(30,15){3}
\put(29,2){$e_i$}
\put(14,15){$e_{i+1}$}
\put(69,7){1}
\put(76.5,7){2}
\put(80,12.25){3}
\put(73,12.5){4}
\put(70.5,1){$e_i$}
\put(59,12.5){$e_j$}
\end{overpic}
\caption{Details of the overlapping ribbon at a fold (left) and at a crossing(right).}
\label{fig:basic-geom-const}
\end{figure}
\end{center}

Since $K$ is a regular polygonal knot diagram, we know that  $\R^2\setminus K$ consists of a finite number of disjoint bounded regions $A_1,\dots, A_n$.  We can assume that $K$ does not intersect the interior of any of these regions. 
We next look at the double points of $\kw$ in each region $A_i$, as well as the unbounded component outside of $K$. For each of the regions $A_i$ we relabel the $k$ vertices of $\bdry A_i$ by $v_{i1}, v_{i2}, \dots, v_{ik}$ as we travel in a counter-clockwise direction around $\bdry A_i$. We denote the fold angle at each vertex $v_{ij}$ by $\theta_{ij}$. As we walk in the counter-clockwise direction about $\bdry A_i$, we either turn to the left at a vertex giving a convex corner, or we turn to the right at a vertex giving a concave corner.  Figure~\ref{fig:ribbon-exterior} shows such a region $A_i$, with a concave corner at vertex $v_{i5}$.    

We observe that at each convex corner of $\bdry A_i$, the ribbon may have a fold or a crossing. (We happen to have chosen a fold at each convex corner in Figure~\ref{fig:ribbon-exterior}.)  We assumed that $K$ does not intersect the interior of $A_i$, hence at every concave corner, there must be a fold.  This means that the double-points of ribbon which are in the {\em interior} of $ A_i$ consist of
\begin{compactitem}
\item a rhombus (at a convex corner),
\item two disjoint triangles (at a concave corner). See Figure~\ref{fig:ribbon-exterior}.
\end{compactitem}

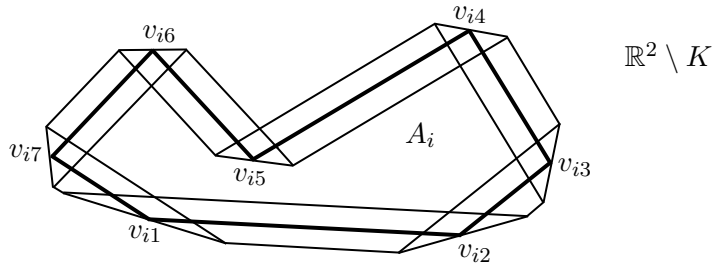
\begin{figure}[htbp]
\begin{center}
\begin{overpic}{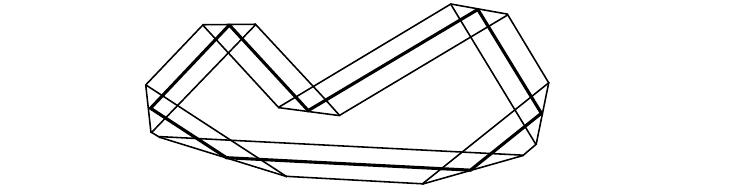}
\put(57,12){$A_i$}
\put(80,20){$\R^2\setminus K$}
\put(28,2){$v_{i1}$}
\put(62.5,0){$v_{i2}$}
\put(73,9){$v_{i3}$}
\put(62,25){$v_{i4}$}
\put(39, 8){$v_{i5}$}
\put(29.5,23){$v_{i6}$}
\put(15.5,10.5){$v_{i7}$}
\end{overpic}
\caption{In one view, the dark line represents the parts of $K$ belonging to $\bdry A_i$. In another view, the dark line represents the parts of of $K$ which are adjacent to the unbounded component of $\R^2\setminus K$.}
\label{fig:ribbon-exterior}
\end{center}
\end{figure}

We now re-use Figure~\ref{fig:ribbon-exterior} to give a different view. Here, we view the image as representing the pieces of the ribbon which are adjacent to the unbounded component of $\R^2\setminus K$. Let the subset of $K$ corresponding to this case be denote by $K_c$. Here, the vertices on the convex hull of $K_c$ are convex corners, and so have folds of the ribbon (not crossings).  The other vertices of $K_c$ will have either folds or crossings. Thus, the double-points of ribbon which are {\em exterior} to $K_c$ and adjacent to the unbounded component of $\R^2\setminus K$ consist of
\begin{compactitem}
\item two disjoint triangles (at a convex corner),
\item a rhombus (at a concave corner).
\end{compactitem}

In each of the regions $A_1, \dots, A_n$, and the unbounded component of $\R^2\setminus K$, there are a finite number of double-point regions of ribbon (rhombi and triangles) and a finite number of single point regions of ribbon (trapezoids).  In each region, we can compute the distance between
\begin{compactitem}
\item two rhombi,
\item pairs of disjoint triangles,
\item rhombi and triangles,
\item triangles and those trapezoids which are disjoint from them.
\end{compactitem}

We can then take the minimum of all of these distances over all the regions, which we will denote by $d$. (Note that if $w=0$, then finding $d$ is equivalent to finding the minimum of the distances between (a) each pair of vertices and (b) each vertex and each edge.)  As we think about $d$, we recall that the vertices of the knot diagram are fixed, but the vertices of the triangles and rhombi (and hence the trapezoids) all depend smoothly on the width and fixed fold-angles. (See formulas for the length of a fold line in Construction~\ref{const:folded-ribbon} and the side-length and diagonal-length of rhombi in Proposition~\ref{prop:basic-geom} (4) and (5).)  
We deduce that for some $\varepsilon>0$, we can continuously increase the width to $w_0+\varepsilon$, such that for each of these widths, the minimum distance $d$ will remain positive. This means that the  folded ribbon knot will have only single and double-points, and $[w_0,w_0+\varepsilon)\in \mathcal{W}$. Since $0\in\mathcal{W}$, this implies there are nonzero elements of $\mathcal{W}$.

To finish our argument, we pick any nonzero $w_0\in\mathcal{W}$. For any width $w$ with $0\leq w<w_0$, the ribbon is narrower, and so the double-point regions are smaller, hence $w\in\mathcal{W}$. Find $\varepsilon>0$ as described above, so that $[w_0,w_0+\varepsilon)\in \mathcal{W}$. Take $\varepsilon_1$ to be the minimum of $\varepsilon$ and $w_0$. Then we know that $(w_0-\varepsilon_1, w_0+\varepsilon_1)\subset \mathcal{W}$. Thus there is an $M>0$ such that $[0,M)\subset\mathcal{W}$.
\end{proof}

\begin{corollary}\label{cor:allowed-width} Given any regular polygonal knot diagram $K$ and folding information $F$, there is a constant $M>0$ such that an allowed folded ribbon knot $K_{w,F}$ exists for all widths $w<M$. 
\end{corollary}

\begin{proof}
Find constant $M>0$ from Proposition~\ref{prop:RibbonConstruct}. Construct the corresponding folded ribbon knot $\kwf$ for some width $w<M$. Since the folded ribbon knot only has single and double points, Definition~\ref{def:allowed} is trivially satisfied.
\end{proof}

\begin{remark}\label{rmk:allowed-width}
What about irregular knot diagrams? Recall that there is an irregular knot diagram for an unknot with two overlapping edges. In this case, there is a corresponding allowed folded ribbon unknot for all widths. Suppose $K$ is an irregular knot diagram with more than two edges, then it can be perturbed to give a regular knot diagram $K'$. Since $K'$ has an allowed folded ribbon knot for small enough widths, we can deduce that $K$ has an allowed folded ribbon knot as well. 
\end{remark}
\section{2-bridge knots on the planar lattice}\label{sect:2-bridge}

Following \cite{Liv}, we remind the reader of the definition of bridge index which was first introduced by H. Schubert \cite{Schu}. Any projection of a knot can be perturbed so that there are a finite number of relative maxima (bridges) and relative minima. It is straightforward to prove that the number of relative minima and maxima are equal.  Then, the minimum number of such maxima taken over all possible projections of a knot is an invariant called the {\em bridge index} of the knot. This invariant gives some idea about how non-planar a knot is.

It can be shown that every 2-bridge knot can be decomposed into two trivial 2-tangles, and hence every 2-bridge knot is a rational knot. (A detailed discussion of the connection between 2-bridge knots, 4-plats and rational knots can be found in \cite{McC} and \cite{BZ} Chapter 12.) Every rational knot has a corresponding continued fraction, and a 2-bridge knot has a continued fraction in the form $[a_1,a_2,\dots,a_m]$ with integers $a_i>0$ and $m$ odd. The representation of this continued fraction as a rational knot diagram is shown in Figure~\ref{fig:Rational-2bridge}. There, the box $a_k>0$ represents $a_k$ half-twists in the counter-clockwise or right-handed direction, and when $a_k<0$ the half-twists are in the clockwise or left-handed direction.  Moreover, for any 2-bridge knot $K$, we see this knot diagram is alternating and reduced. Hence the number of crossings in the diagram, $a_1+a_2+ \dots+a_m$, is the crossing number $\Cr(K)$ (see \cite{Kauf87, Mur87, Thi}).

\begin{figure}[htbp]
\centering
\begin{overpic}{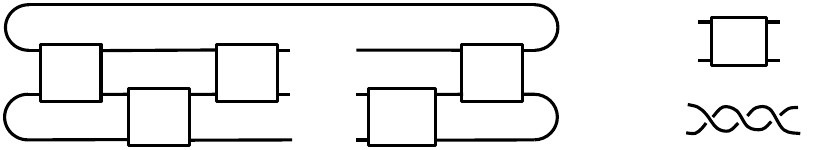}
\put(6,9){$-a_1$}
\put(18, 3){$a_2$}
\put(27.5,9){$-a_3$}
\put(38,1){\dots}
\put(38,6.5){\dots}
\put(38,12){\dots}
\put(46, 3){$a_{m-1}$}
\put(57,9){$-a_m$}
\put(88,12.5){$a_k$}
\put(89,7){$\updownarrow$}
\end{overpic}
\caption{On the left, a 2-bridge knot. On the right, a box with $a_k>0$ represents $a_k$ half-twists in the counter-clockwise direction.}
\label{fig:Rational-2bridge}
\end{figure}
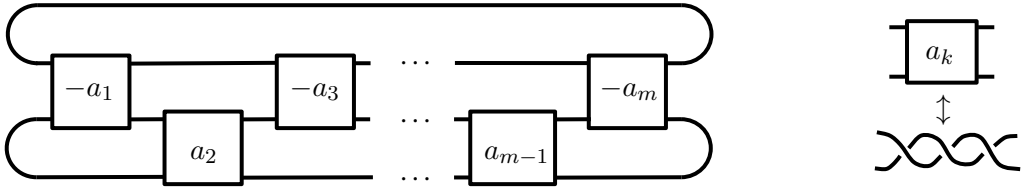

In 2014, Y. Huh, K. Hong, H. Kim, S. No and S. Oh \cite{HHKNO} used a particular embedding of a 2-bridge knot in the cubic lattice to find an upper bound on ropelength in terms of crossing number. We will use their embedding to obtain a similar result for folded ribbonlength. Recall that the {\em cubic lattice} is $(\R\times \Z\times \Z) \cup (\Z\times \R\times \Z) \cup (\Z\times \Z\times \R)$, and the {\em planar lattice} is $(\R\times \Z) \cup (\Z\times \R)$. Suppose we have a knot embedded in the cubic lattice. In this context, an {\em edge} is a line segment of unit length joining two nearby lattice points in the cubic lattice. An edge parallel to the $x$-axis is called an {\em $x$-edge}, similarly for $y$-edges and $z$-edges.
 We can project the knot embedded in the cubic lattice onto the $xy$-plane to get a polygonal knot diagram in the planar lattice. The height in the $z$-axis direction gives the crossing information for the knot diagram. Given a polygonal knot diagram $K$ in the planar lattice, we obtain a corresponding folded ribbon knot $\kw$ of width $w=1$. (To see this, note that the distance between the lattice lines is 1 unit, so $\nicefrac{w}{2}=\nicefrac{1}{2}$.) We have all the ingredients in place to prove Theorem~\ref{thm:2-bridge}: Any 2-bridge knot type $K$ contains a folded ribbon knot $\kw$ such that   $\Rib(\kw) \le 6\Cr(K)-2.$

\begin{proof}[Proof of Theorem ~\ref{thm:2-bridge}]
We follow \cite{HHKNO} closely and give an embedding of $K$ in the cubic lattice in two steps.

Step 1. This is illustrated in Figure~\ref{fig:2bridge-const} (left) for the 2-bridge link with Conway notation $(3,2,1)$. In this figure, the $z$-direction is perpendicular to the page. We set the link to be at height $z=2$, except for the dotted line segments which are at height $z=1$. The vertices with a bold dot indicate the location of one $z$-edge between heights $z=1$ and $z=2$. We can count the number of lattice edges used in the embedding by looking at each crossing as shown on the right in Figure~\ref{fig:2bridge-const}. There are 2 edges in each of the $x$ and $y$-directions at any crossing. The other two strands of the diagram contribute another 4 edges in some combination of  $x$ and $y$-directions. At each crossing, there are also two  $z$-edges at the bold dots. Thus, the cubic lattice length of this embedding is $10\Cr(K)$. If we were to project to the $xy$-plane at this step, we would remove $2\Cr(K)$ $z$-edges. Thus the planar lattice length of the polygonal knot diagram at this stage is $\Len(K)=8\Cr(K)$. However, we can do better!

\begin{figure}[htbp]
    \centering
    \begin{overpic}{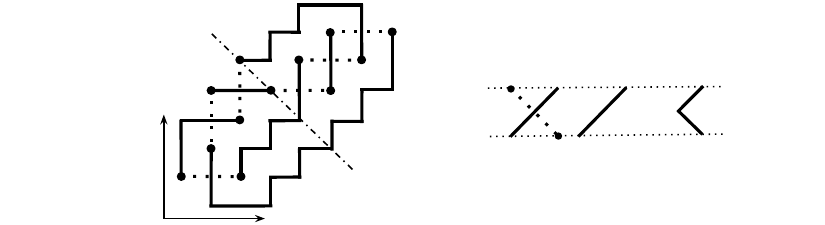}
    \put(48,24){$a$}
    \put(24,2){$b$}
    \put(37,12.5){$c$}
    \put(43,5){$\ell$}
    \put(19,14){$y$}
    \put(32.5,0){$x$}
    \put(77,13){or}
    \end{overpic}
    \caption{On the left an embedding in the cubic lattice of the 2-bridge link with Conway notation $(3,2,1)$. On the right,  a detailed view of the number of edges in the embedding near a crossing.}
    \label{fig:2bridge-const}
\end{figure}
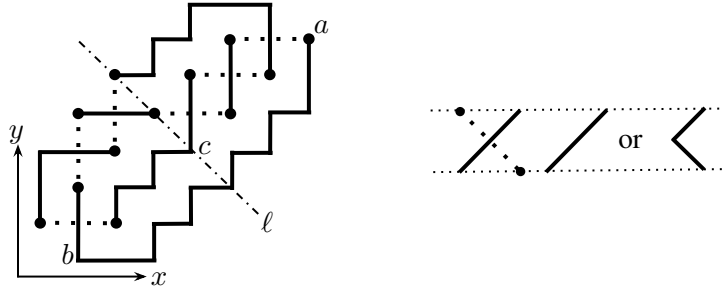

Step 2. We can further reduce the cubic (and hence planar) lattice length by ``folding" the link in half and removing unnecessary edges. We give a brief outline here, again following \cite{HHKNO} closely. First delete the bottom right arc at height $z=2$ between points $a$ and $b$ in Figure~\ref{fig:2bridge-const} (left). This removes $2\Cr(K)$ edges from the embedding. 

Next, cut the remaining link  in two halves  in the middle along the dot-dash line $\ell$. The exact position of $\ell$ depends on whether $\Cr(K)$ is odd or even, but in both cases the knot has been cut in three places.  We then take the half closest to the $x$- and $y$-axes and rotate it  $180^\circ$ around $\ell$ and move that half up to $z$-levels 3 and 4.  The points $a$ and $b$ are then joined together again with a net gain of one $z$-edge. The precise construction again depends on whether $\Cr(K)$ is even or odd. (We omit the details, all of which can be found in \cite{HHKNO} Figures 3 and 4.)  We now join the remaining 3 pairs of cutting points along $\ell$ by adding in three $z$-edges.  Finally, we can remove an extra two $x$- or $y$-edges from the cubic lattice embedding near the point $c$. (The details are shown in \cite{HHKNO} Figure~4.) The cubic lattice length of the embedding is $10\Cr(K)-2\Cr(K)+ 4-2 = 8\Cr(K)+2$. 

We now project the embedding of the knot in the cubic to the $xy$-plane, giving a polygonal knot diagram in the planar lattice.  We can view this polygonal knot diagram as the center of a folded ribbon knot $\kw$ of width $w=1$. The ribbonlength of $\kw$ is then just the length of the knot in the planar lattice. To find the length of the knot in the planar lattice, we simply ignore the $z$-edges in the cubic lattice embedding described above. Hence we find
$$
\Rib(\kw)=\Len(K) \le 8\Cr(K)+2-2\Cr(K)-4= 6\Cr(K)-2.
$$
\end{proof}

This result is a nice example of an infinite family of knots whose folded ribbonlength is bounded above linearly by crossing number.
In fact, 2-bridge (rational) knots and links are an example of a bigger class of knots called Conway algebraic knots. Y. Diao, C. Ernst and U. Ziegler \cite{DEZ} have given a linear growth bound on the minimum ropelength for these knots. We have not attempted to apply their methods here to the folded ribbonlength problem.


\section{Ribbonlength of pretzel, twist, and $(2,q)$ torus knots}\label{sect:PTT} 

We begin by reminding the reader that pretzel, twist, and ($2,p$) torus knots each contain at least one region with a series of half-twists. Following the construction of 2-bridge knots in Section~\ref{sect:2-bridge}, we let $k>0$ represent $k$ half-twists in the counterclockwise direction, similarly $k<0$ represents $k$ half-twists in the clockwise direction. 
Recall that  a {\em pretzel knot} is constructed by connecting the adjacent ends of three vertically parallel twisting regions, consisting of $p,q,$ and $r$ half-twists. This is shown on the  left of Figure~\ref{fig:PTT} for the $(3, -2, 3)$ pretzel knot, denoted by $P(3,-2,3)$. \begin{remark}\label{rmk:pretzel}
We review some facts about pretzel links (see for instance \cite{Crom, JohnHen}).
\begin{compactenum}
\item A pretzel link $P(p,q,r)$ is a knot if and only if at most one of $p$, $q$ and $r$ are even. 
\item A pretzel link $P(p,q,r)$ is alternating if $p,q$ and $r$ have the same sign.
\item The mirror image $P^m(p,q,r)$ of a pretzel link is equivalent to $P(-p,-q,-r)$.
\item A pretzel link $P(p,q,r)$ is equivalent to a pretzel link with any permutation of $p$, $q$ and $r$. For example, $P(p,q,r) \cong P(p,r,q) \cong P(q,r,p)$.
\item The crossing number of a pretzel link can be hard to find in general. We have the following specific cases:
\begin{compactenum} 
\item If $p,q,r$ have the same sign, then $P(p,q,r)$ has a reduced alternating diagram, so $\Cr(P(p,q,r))=|p|+|q|+|r|$.
\item  The crossing number of a pretzel link $P(-p,q,r)$ where  $p,q,r\ge2$ is given by $\Cr(P(-p,q,r))=p+q+r$ (see \cite{LeeJin}).
\end{compactenum}
\end{compactenum}
\end{remark}

 \begin{center}
\begin{figure}[htbp]
\begin{overpic}{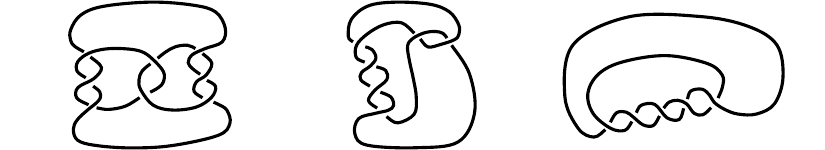}
\put(37,7){twists}
\put(56.5,13){clasp}
\end{overpic}
\caption{On the left, the $P(3, -2, 3)$ pretzel knot; in the center, a $T_4$ twist knot with 4 half-twists; on the right a $T(2,5)$ torus knot. }
\label{fig:PTT}
\end{figure}
\end{center}
 
A {\em twist knot}, denoted $T_n$, is an alternating knot consisting of $n$ half-twists and a clasp region. The clasp region consists of 2 half-twists made in the opposite direction of the twist region, and rotated 90 degrees. A twist knot with 4 half-twists, denoted $T_4$, is shown in the center in Figure~\ref{fig:PTT}.  Since this knot diagram of $T_n$  is alternating and reduced, we can deduce that the crossing number $\Cr(T_n)=n+2$.  We also note that twists knots are pretzel knots. There are many ways to see this, the simplest is that $T_n=P(n,1,1)$. This is illustrated in Figure~\ref{fig:PTT} where $T_4=P(4, 1,  1)$. 

Finally, a {\em $(p,q)$ torus knot}, denoted by $T(p,q)$,  is a knot which can be embedded onto a torus such that the knot traverses $p$ times around the longitude of the torus and $q$ times around the meridian of the torus. We always assume $p>0$; whether $q>0$ or $q<0$ determines the direction of winding around the meridian. A $T(2,5)$ knot is shown on the right in Figure~\ref{fig:PTT}.  There is quite a lot known about torus knots  (see for instance \cite{Adams,Crom}), and we recall several useful facts about them.
\begin{compactenum}
\item $T(p,q)$ is a knot if and only if $p$ and $q$ are coprime.
\item $T(p,q)$ is an unknot if and only if either $p$ or $q$ is $\pm1$.
\item The crossing number of a $T(p,q)$ knot is $\Cr(T(p,q))=\min\{p(q-1), q(p-1)\}$. Hence $\Cr(T(2,q))=q$.
\item To simplify our arguments later, we assume that all torus knots and links have $p,q\ge 2$.
\end{compactenum}

In the remainder Section~\ref{sect:PTT}, we will use a particular construction of a half-twist region to show that the ribbonlength of $(2,q)$ torus, twist, and certain pretzel knots are bounded above linearly by crossing number.

\subsection{Constructing $p$ half-twists}\label{section:twists}

In this and subsequent sections, we describe constructions of knots using folded ribbons. Many of these can be hard to visualize, and our pictures are far from perfect. So we strongly encourage the reader to cut up some paper strips and fold along as we describe our constructions.

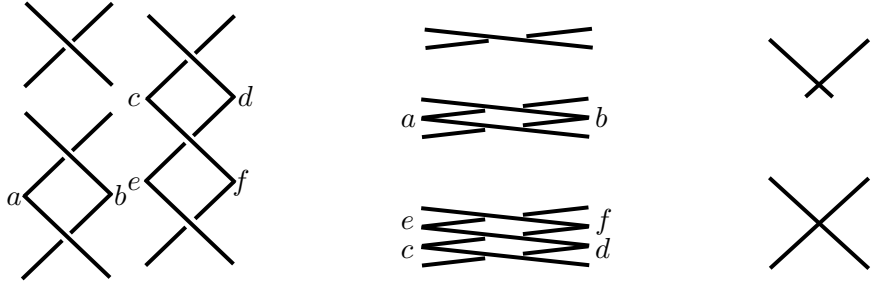
\begin{figure}[htbp]
   \begin{overpic}{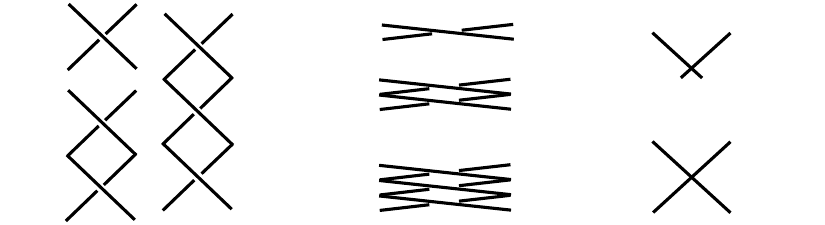}
    \put(6.5,7.5){$a$}
    \put(16.75,7.5){$b$}
    \put(18,17){$c$}
     \put(28.5,17){$d$}  
     \put(18,9){$e$}
     \put(28,9){$f$}
     \put(44,15){$a$}
     \put(62.5,15){$b$}
     \put(44,2.5){$c$}
     \put(62.5,2.5){$d$}  
     \put(44,5.25){$e$}
     \put(62.5,5.25){$f$}
    \end{overpic}
    \caption{On the left, parts of a knot diagram with 1, 2 and 3 half-twists. In the center, moving the knot so that in the diagram, the crossings in the half-twists lie on top of one another. On the right, the view of the corresponding knot diagram.}
    \label{fig:twists}
\end{figure}

We can imagine manipulating the part of a knot diagram containing $p$ half-twists such that all of the crossings lie atop one another.  To see this manipulation look at Figure~\ref{fig:twists}. On the left we see diagrams with one, two and three half-twists. Now, consider the diagram with two half twists. Imagine folding the diagram in half at the points $a$ and $b$. The edges adjacent to $a$ will lie on top of each other, as will the edges adjacent to $b$.  In the center of Figure~\ref{fig:twists} we see a side view of this folding. We do the same thing with the diagram with three twists. This time we imagine folding the diagram up at points $c$ and $d$, then folding it down at points $e$ and $f$. At this point, the the edges adjacent to vertices $c$ and $f$ lie on top of each other. (Note that vertex $c$ and $f$ share an edge.)  Similarly, the edges adjacent to vertices $d$ and $e$ lie on top of each other. This folding construction has created non-regular knot diagrams of the half-twists, and these are shown on the right in Figure~\ref{fig:twists}.  The top image on the right in Figure~\ref{fig:twists} shows the diagram for two half-twists, and the bottom image is the diagram for both one and three half-twists. 

We will now describe how to construct a part of a folded ribbon knot $\kw$ with a series of $p$ half-twists as described above.

\begin{const}[Half-twists]\label{const:twists}
Let $AB$ and $CD$ be two strands of a ribbon of width $w$, as shown on the left of Figure~\ref{fig:Counter}. If $p>0$, the half-twists are in a counterclockwise direction. So, begin by crossing $AB$ over $CD$  at an angle of $\nicefrac{\pi}{2}$  so that $A$ is to the southwest, $C$ is to the southeast, $B$ is to the northeast, and $D$ is to the northwest. The creates one half-twist, and is illustrated by the two leftmost images in Figure~\ref{fig:Counter}. To create a second half-twist, we need to mimic the folding described above. Thus we fold $C$ over the intersection of $AB$ and $CD$ so that it rests along $D$, as shown in the third image from left in Figure~\ref{fig:Counter}. Then, fold $A$ over $CD$ so that it rests along $B$ (see Figure~\ref{fig:Counter}, right). 

To create a third half-twist, fold $C$ back over $AB$ so that it lays to the southeast. Similarly, fold $A$ back over $CD$ so that it lays to the southwest. Each time we construct a half-twist we need to fold over both the $A$ and $C$ ends of the ribbon.  The reader is encouraged to fold the ribbon as described. Then, in one hand hold ends $B$ and $D$, and in the other hand hold ends $A$ and $C$. By gently moving the ends apart, the reader will see the half-twists.

\begin{figure}[htbp]
    \centering
    \begin{overpic}{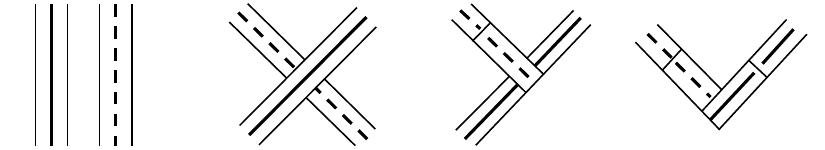}
    \put(1.75,0){$A$}
    \put(1.75,16.5){$B$}
    \put(16.5,0){$C$}
    \put(16.5,16.5){$D$}
    \put(28,0){$A$}
    \put(26.5,17){$D$}
    \put(45,-0.5){$C$}
    \put(44.5,17){$B$}
    \put(54,0){$A$}
    \put(53.5,17){$D$}
    \put(60,15.5){$C$}
    \put(70.5,16.5){$B$}
    \put(77,14.5){$D$}
    \put(83,12.5){$C$}
    \put(94,7.5){$A$}
    \put(96.5,15.5){$B$}
    \end{overpic}
    \caption{Adding two counterclockwise half-twists to a pair of strands.}
    \label{fig:Counter}
\end{figure}
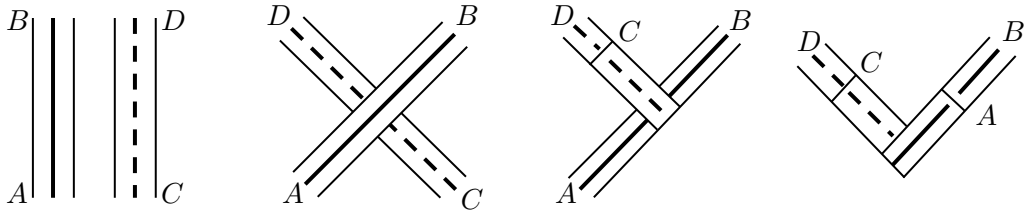

Keep repeating this folding process to produce a series of $p$ half-twists. Note that if $p$ is even, $A$ and $C$ will rest along $B$ and $D$, respectively. If $p$ is odd, $A$ and $C$ will point in the opposite directions of $B$ and $D$, respectively.

If $p<0$, the twists are in a clockwise direction. So, begin by crossing $CD$ over $AB$ at an angle of $\nicefrac{\pi}{2}$ (with the same labels as shown in Figure~\ref{fig:Counter} second from the left), creating one half-twist. To create more half-twists proceed in a similar way as the $p>0$ case. \qed
\end{const}

\begin{remark}\label{rmk:squares}
We observe that the $p$ half-twists constructed in Construction~\ref{const:twists} creates $2p$ squares of folded ribbon, each with sidelength $w$, and four extended ends. Letting $w=1$, we see that this construction yields a ribbonlength of $2p$ plus the length of any connections made between the ends. 
\end{remark}

Our construction of $p$ half-twists yields a piece of folded ribbon that can easily be manipulated to create a variety of folded ribbon knots. We now use this to create folded ribbon $(2,q)$ torus, pretzel, and twist knots.

\subsection{$(2,q)$ torus knots}\label{sect:2q-torus}

In this section, we give a construction of a folded ribbon knot $\kw$ where $K$ is a $(2,q)$ torus knot and calculate ribbonlength in relation to crossing number. Note that a $(2,q)$ torus knot is constructed by taking $q$ half-twists then joining the ends appropriately. For the sake of simplicity, we will only consider the cases in which $q\geq 3$ is odd, thereby guaranteeing that $T(2,q)$ is a non-trivial knot. 

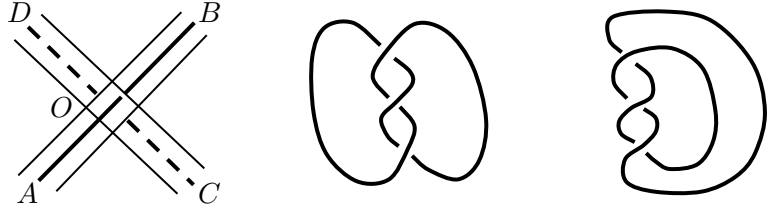
\begin{figure}[htbp]
    \centering
    \begin{overpic}{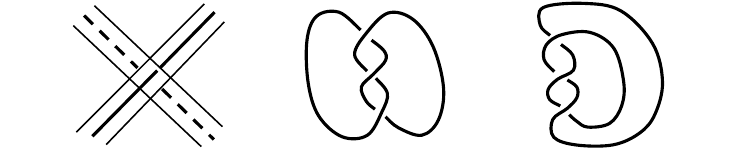}
    \put(10,-0.5){$A$}
    \put(9,18.5){$D$}
    \put (29, -0.5) {$C$}
    \put(29,18.5){$B$}
    \put(13.5,8.5){$O$}
    \end{overpic}
    \caption{On the left, $q$ half-twists in a ribbon. In the center and right, we connect $A$ to $D$ and $B$ to $C$ to create a $T(2,q)$ torus knot for $q=3$.}
    \label{fig:Torus-2p-connect}
\end{figure}

\begin{const}[Folded ribbon $T(2,q)$ knots]\label{const:2p-torus}
Assume the ribbon is of width $w$ and $q\ge3$, odd.

Step 1. As outlined in Construction~\ref{const:twists}, make a series of $q$ half-twists in the counterclockwise direction between strands of ribbon $AB$ and $CD$.  Since $q$ is odd, after the $q$ half-twists, $A$ and $C$ will point in opposite directions of $B$ and $D$ respectively. 

Step 2. As shown in Figure~\ref{fig:Torus-2p-connect}, we need to connect $A$ to $D$, and $B$ to $C$ in order to have a $T(2,q)$ torus knot.  In this figure, point $O$ is the multiple point of the knot diagram.  As the twists are in a counter-clockwise direction, we observe that edge $OA$ is over $OC$. We first fold $A$ up at $O$ so it rests along $D$, as illustrated by the image second from left in Figure~\ref{fig:TorusMinimization}. When we do this, we observe the center square of ribbon containing point $O$ is replaced with a fold from angle $\nicefrac{\pi}{2}$, and thus there is no net gain or loss of ribbon. Second, we fold $C$ up so that it rests along $B$, as in the third image from left in Figure~\ref{fig:TorusMinimization}. Finally, we glue the corresponding ends together ($A$ to $D$, and $B$ to $C$) which creates two folds from angle $0$. We then shrink the connections until they rest along the boundary of the center square.  The complete folded ribbon $T(2,q)$ torus knot is shown on the right of Figure~\ref{fig:TorusMinimization}.
\qed
\end{const}
\begin{figure}[htbp]
    \centering
    \begin{overpic}{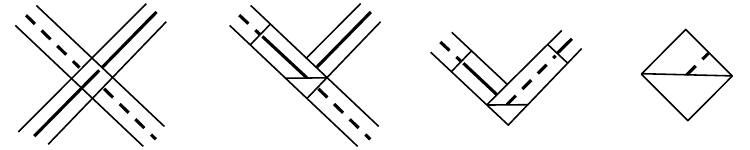}
    \put (2.5,-0.5) {$A$}
    \put (21,19) {$B$} 
    \put (21,-0.5) {$C$}
    \put (1,18) {$D$}
    \put(5.5,8.5){$O$}
    \put (31.5,11) {$A$}
    \put (49, 19) {$B$}
    \put (50, -0.5) {$C$}
    \put (29,18) {$D$}
    \put (58, 7.5) {$A$}
    \put (76, 15) {$B$}
    \put (75, 8.5) {$C$}
     \put (56, 14.5) {$D$}
    \put (82, 13) {$A, D$}
    \put (95, 13) {$B, C$}
    \end{overpic}
    \caption{In the left three images we fold the ribbon to identify $A$ to $D$, and $C$ to $B$. On the right, the completed folded ribbon $T(2,q)$ knot.}
    \label{fig:TorusMinimization}
\end{figure}
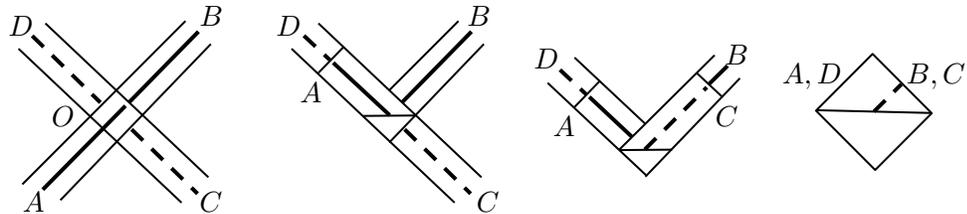

Note that if we were to wiggle and stretch the ribbon, we would quickly see that $K$ is indeed equivalent to a $(2,q)$ torus knot.   We are now ready to prove Theorem~\ref{thm:2p-torus}: any  $(2,q)$ torus knot type $K$ contains a folded ribbon knot $\kw$ such that $\Rib(\kw) \le 2q=2\Cr(T(2,q)).$

\begin{proof}[Proof of Theorem~\ref{thm:2p-torus}]
Let the width $w=1$. In Construction~\ref{const:2p-torus} Step 1, the $q$ half-twists give $2q$ unit-squares of ribbon. Hence the folded ribbonlength is $2q$ plus the length of any connections. In Construction~\ref{const:2p-torus} Step 2, we fold $A$ to $D$, and the unit square of ribbon becomes a fold from angle $\nicefrac{\pi}{2}$. From Corollary~\ref{cor:triangles}, this fold has folded ribbonlength 1. Thus this maneuver does not change the ribbonlength.  Similarly, the folded ribbonlength is unchanged when we fold $C$ to $B$. The folded ribbonlength of this particular folded ribbon $T(2,q)$ knot  is $2q=2\Cr(T(2,q))$. 
\end{proof}

\begin{remark}
Construction~\ref{const:2p-torus} uses $2q+2$ sticks to construct a $(2,q)$ torus knot.  (There are $2q$ sticks from the $q$-half-twists and the extra 2 sticks comes from identifying $A$ to $D$ and $B$ to $C$.) Thus, the resulting folded ribbon torus knot is always a topological annulus. 
\label{rmk:2p-torus}
\end{remark}

Theorem~\ref{thm:2p-torus} means the folded ribbonlength for a folded ribbon $T(2,q)$ torus knot has a linear upper bound in crossing number. Previously, Kennedy {\em et al.} \cite{Den-FRS, KMRT} showed that $\Rib(T(2,q)) \leq q \cot(\nicefrac{q}{2})$ for $q\geq 7$. This is a quadratic upper bound on the folded ribbonlength for $T(2,q)$ torus knots.  Tian \cite{Tian} proved that for all $p$, $q$, the $T(p,q)$ torus knots have minimum folded ribbonlength bounded above by $8\Cr(K)$.  Theorem~\ref{thm:2p-torus} improves Tian's bound for $T(2,q)$ knots.  Of interest is the special case of the trefoil knot $T(2,3)$. Both Kauffman \cite{Kauf05} and  Kennedy {\em et al.} \cite{KMRT} used a 5-stick trefoil knot to show that the folded ribbonlength is bounded above by $5\cot(\nicefrac{\pi}{5})\leq 6.882$. However, Theorem~\ref{thm:2p-torus} and Remark~\ref{rmk:2p-torus} shows that by using 8 sticks, we can lower this bound down to~$6$.

\subsection{Pretzel and twist knots}\label{sect:pretzel}
In this section we give the construction of a folded ribbon pretzel knot $\kw$ for $K$ a pretzel knot $P(p,q,r)$. We then calculate the ribbonlength of $\kw$ with respect to $p,q,$ and $r$. 

 \begin{figure}[htbp]
    \centering
    \begin{overpic}{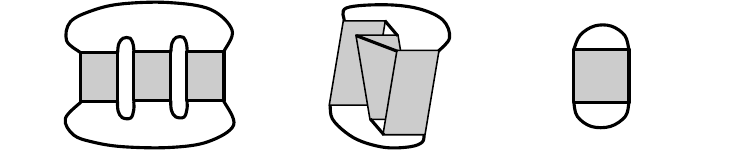}
    \put (12,9) {$p$}
    \put (19.5,9) {$q$}
    \put (27,9) {$r$}
    \put (46,10) {$p$}
    \put (50,9) {$q$}
    \put (55,8) {$r$}
    \put (79.5,9) {$r$}
    \end{overpic}
    \caption{An outline of Construction~\ref{const:pretzel}, where we reduce the folded ribbonlength by arranging the twist regions of the pretzel knot $P(p,q,r)$, accordion style, on top of one another.}
    \label{fig:general-pretzel}
\end{figure}
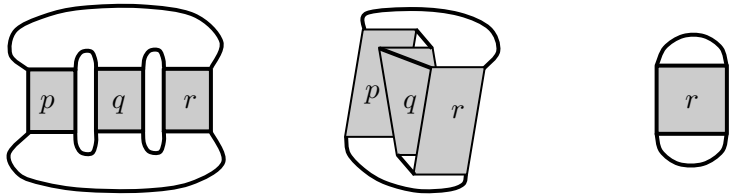

\begin{const}[Folded ribbon pretzel knots]\label{const:pretzel}
A schematic picture of a $P(p,q,r)$ pretzel knot is shown on the left in Figure~\ref{fig:general-pretzel}. We first give an outline of our construction.  Following Construction~\ref{const:twists}, take a ribbon of width $w$, and make a series of $p$, $q$, and $r$ half-twists between three pairs of strands of ribbon in the direction appropriate for the sign(s) of $p$, $q$ and $r$. Next, fold the twist regions accordion-style on top of one another as shown in the center of Figure~\ref{fig:general-pretzel}. On the right of Figure~\ref{fig:general-pretzel}, we see that this arrangement allows us to reduce the lengths of many of  the connections between the twist regions to 0.  As before, the reader is encouraged to cut strips of paper and fold the constructions as they read along.

The precise details of the construction will depend upon the sign of $p$, $q$ and $r$, and whether they are even or odd. Since this construction is just for folded ribbon pretzel knots, we can use the facts in Remark~\ref{rmk:pretzel} to reduce the number of cases that we need to consider. Throughout this construction, we will repeatedly use Corollary~\ref{cor:triangles}: there is 1 unit of folded ribbonlength corresponding to a fold from angle $\nicefrac{\pi}{2}$. 

   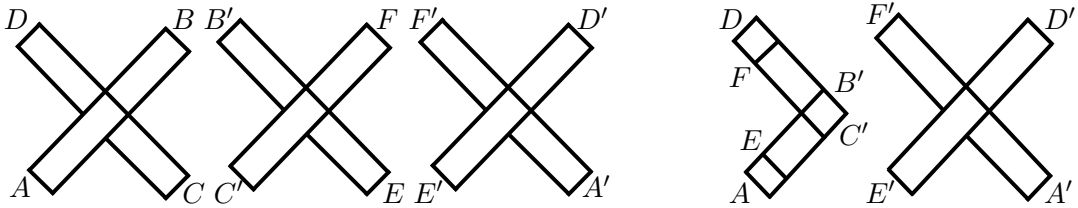
\begin{figure}[htbp]
        \centering
        \begin{overpic}{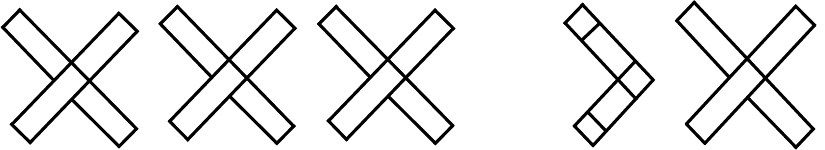}
        \put (-1,16.5) {$D$}
        \put (-0.5,0.5) {$A$}
        \put (15,16.5) {$B$} 
        \put (16,0) {$C$}
        \put (18,16.5) {$B'$}
        \put (19,0) {$C'$}
        \put (35,0) {$E$}
        \put (34.4,16.5) {$F$}
        \put (37.5,16.5) {$F'$}
        \put (38,0) {$E'$}
        \put (53.5,16.5) {$D'$}
        \put (54,0.5) {$A'$}
        \put (68, 0) {$A$}
        \put (67, 16.5) {$D$}
        \put (68, 11) {$F$}
        \put (69, 5) {$E$} 
        \put (78.5, 5.5) {$C'$}
        \put (78, 10.5) {$B'$}
        \put (81, 0) {$E'$}
        \put (81,17) {$F'$}
        \put (98,0) {$A'$}
        \put (98,16.5) {$D'$}
        \end{overpic}
        \caption{On the left, labeling the strands of the three twist regions in Case 1. On the right, we arrange the $q$ half-twists atop the $p$ half-twists as in Figure~\ref{fig:general-pretzel}, with the $r$ half-twists to the far right.}
        \label{fig:Case 1.1}
    \end{figure}

{\bf Case 1.} Let $K=P(p,q,r)$ be a knot such that $p,q,r>0$, and $p,q,r$ all odd. In the region of $\kw$ representing $p$ half-twists, label the southeast end $C$, the southwest end $A$, the northwest end $D$, and the northeast end $B$. In the region of $\kw$ representing $q$ half-twists, label the northwest end $B'$, the northeast end $F$, the southwest end $C'$, and the southeast end $E$. Finally, in the region of $\kw$ representing $r$ half-twists, label the northwest end $F'$, the southwest end $E'$, the northeast end $D'$, and the southeast end $A'$. See  Figure~\ref{fig:Case 1.1}, left, for the complete labeling. To construct our pretzel knot $K$, we just need to identify the ends $A, A'$, then $B,B'$,  \dots,  then $F,F'$
    
We now arrange our regions with half-twists on top of one another as in  in Figure~\ref{fig:general-pretzel} center. We first place the region with $q$ half-twists atop the region with $p$ half-twists such that $C'$ and $B'$ rest along $C$ and $B$, respectively. We can imagine gluing these ends together creating a fold from angle $0$, then shrinking the connections so that their unions lay on the boundary of the center squares. The construction after this shrinking can be seen in the right of Figure~\ref{fig:Case 1.1}: note that these shortened connections add no extra ribbonlength to the construction. We next place the region with $r$ half-twists atop the region with $q$ half-twists  such that $F'$ and $E'$ rest along $F$ and $E$, respectively. Similarly, we can imagine gluing these ends together creating a fold from angle $0$, and then shrinking these new connections to lay directly on the boundary of the center squares (see Figure~\ref{fig:Case 1.2} left).  Again the shortened connections add no extra length to the construction. 
    
        \begin{figure}[htbp]
        \centering
        \begin{overpic}{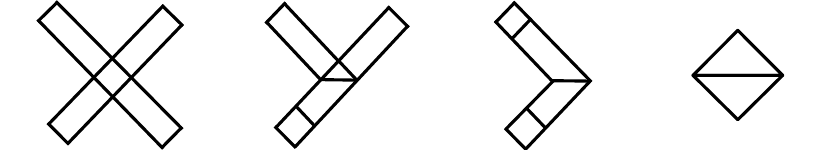}
        \put (4, 0) {$A$}
        \put (3.5,17) {$D$}
        \put (21,0) {$A'$}
        \put (21,17) {$D'$}
        \put (32, 0) {$A$}
        \put (38.5, 2) {$A'$}
        \put (31, 17) {$D$}
        \put (48.5, 16.5) {$D'$}
        \put (60,-0.5) {$A$}
        \put (66.5, 1.5) {$A'$}
        \put (58.5,17) {$D$}
        \put (65, 16) {$D'$}
        \put (81, 12) {$D, D'$}
        \put (80.5, 4) {$A, A'$}
        \end{overpic}
        \caption{On the left, we arrange the $r$ half-twists atop the $p$ and $q$ half-twists in Case 1. In the center images, we fold the parts of the ribbon to identify $A$ to $A'$, and $D$ to $D'$. On the right, the completed folded ribbon $P(p,q,r)$ knot.}
        \label{fig:Case 1.2}
    \end{figure}
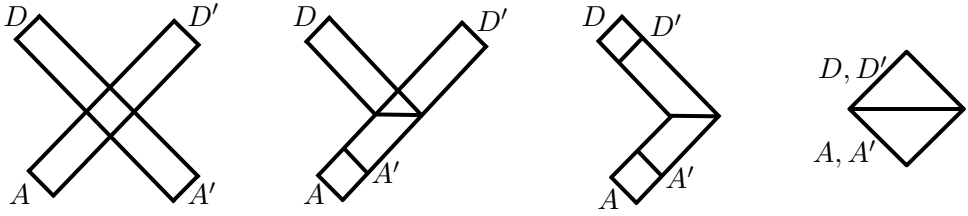

All that remains is to connect $A$ to $A'$, and $D$ to $D'$ in Figure~\ref{fig:Case 1.2} (left). To do this we repeat the ideas seen in Construction~\ref{const:2p-torus}. We first fold $A'$ up over $AD'$ with a fold from angle $0$. We then fold $A'$ down to the left with a fold from angle $\nicefrac{\pi}{2}$ and fold line along the diagonal of the center square, so that $A'$ rests along $A$ as shown in Figure~\ref{fig:Case 1.2}, second image from left. (This action adds 1 unit of folded ribbonlength corresponding to the fold from angle $\nicefrac{\pi}{2}$.)  We similarly fold $D'$ so that it rests along $D$, illustrated in Figure~\ref{fig:Case 1.2}, third image from left. Once more, we glue these ends together ($A$ to $A'$ and $D$ to $D'$) creating two folds from angle $0$, and then shrink the subsequent connections so that they rest along the boundary of the center square.  The completed construction is shown on the right of Figure~\ref{fig:Case 1.2}.

{\bf Case 2.}  Let $K=P(p,q,r)$ be a knot such that $p,q>0$, $r<0$, and $p,q,r$ all odd. This case is identical to Case 1, except that the region of $\kw$ representing $r$ half-twists should be constructed in the clockwise direction.

{\bf Case 3.} Let $K=P(p,q,r)$ be a knot such that $p,q,r>0$, and $p,r$ odd, and $q$ even. Label the regions of $\kw$ representing $p$ and $r$ half-twists in a similar way to Case 1 as shown in Figure~\ref{fig:Case3-1}. There, the $q$ half-twists are labeled as follows: the bottom northwest end is $B'$, the bottom northeast end is $F$, the top northwest end is $C'$, and the top northeast end is $E$. As in Case 1, we need to identify the appropriate ends to create $K$.

  \begin{figure}[htbp]
        \centering
        \begin{overpic}{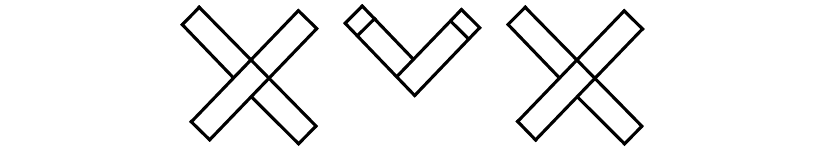}
        \put (20.5,16.5) {$D$}
        \put (21.5,0.5) {$A$}
        \put (37,16.5) {$B$} 
        \put (37.5,0) {$C$}
        \put (40.5,16.5) {$B'$}
        \put (41.5,11) {$C'$}
        \put (57,11.5) {$E$}
        \put (57,16.5) {$F$}
        \put (60.5,16.5) {$F'$}
        \put (61,0) {$E'$}
        \put (77,16) {$D'$}
        \put (77,0.5) {$A'$}
             \end{overpic}
        \caption{Labeling the strands of the three half-twist regions in Case 3.}
        \label{fig:Case3-1}
    \end{figure}
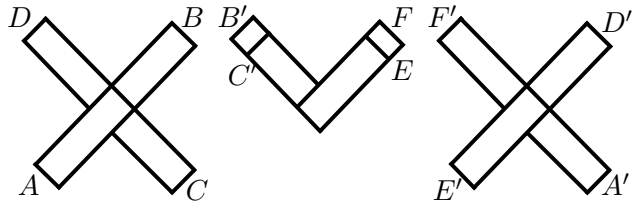

In order to easily identify ends, we first to fold $C$ over the center squares, then up to the right so that it rests atop $B$. (Note that this action adds 1 unit of folded ribbonlength corresponding to a fold from angle $\nicefrac{\pi}{2}$.)  Place the region with $q$ half-twists atop the region with $p$ half-twists such that $C'$ and $B'$ rest along $C$ and $B$, respectively.  When we do this, $C$ and $C'$ are on the inside of $B$ and $B'$. We then glue $C$ and $C'$ together creating a fold from angle $0$, and shrink their connection until it lies on the boundary of the center squares. We then similarly glue $B$ and $B'$ together and shrink the connection. The final result is shown on the left in Figure~\ref{fig:Case3-2}. 

  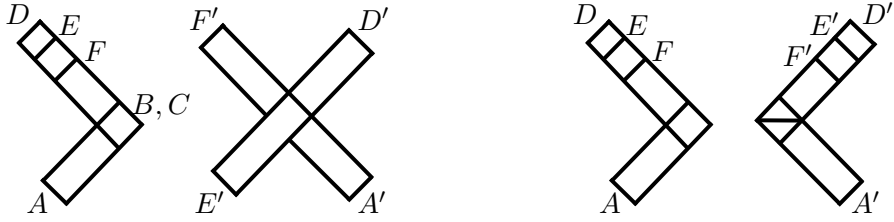
\begin{figure}[htbp]
        \centering
        \begin{overpic}{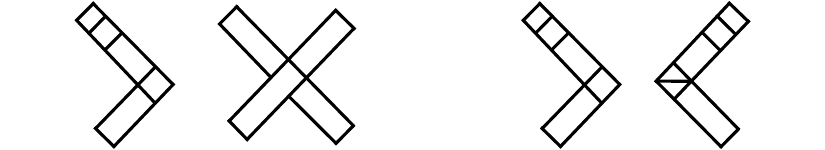}
        \put (10,-0.5) {$A$}
        \put (20,9) {$B,C$} 
        \put (8,17.5) {$D$}
        \put (13,16.5) {$E$}
        \put (15.5,14) {$F$}
        \put (25.5,17) {$F'$}
        \put (26,-0.5) {$E'$}
        \put (41.5,17) {$D'$}
        \put (41.5,-0.5) {$A'$}
        \put (64.5,-0.5) {$A$}
        \put (62,17.5) {$D$}
        \put (67,16.5) {$E$}
        \put (69.5,14) {$F$}
        \put (82,13.5) {$F'$}
        \put (84.5,16.5) {$E'$}
        \put (89.5,17.5) {$D'$}
        \put (88.5,-0.5) {$A'$}
         \end{overpic}
        \caption{Case 3 more.}
        \label{fig:Case3-2}
    \end{figure}

At this point, if we look along the northwest branch of the ribbon (in Figure~\ref{fig:Case3-2} left), we see $D$ on the bottom, then $E$ in the middle, then $F$ on the top. We next look at the region with $r$ half-twists and make adjustments so we can identify the corresponding ends. 
We fold $E'$ back to lie along $D'$, and then fold $F'$ so it lies atop $E'$. This second maneuver adds 1 unit of folded ribbonlength corresponding to a fold from angle $\nicefrac{\pi}{2}$. The final result is shown on the right in Figure~\ref{fig:Case3-2}.  

We now place the region with $r$ half-twists atop the region with $q$ half-twists, such that $F', E', D'$ rest along $F, E, D$, and $A'$ rests along $A$. Here, $F, F'$ are inside $E, E'$, which in turn are inside $D,D'$. At this point, we identify $F$ and $F'$ creating a fold from angle $0$, and shrink the connection until it lies along the edge of the center squares. Similarly, for $E, E'$, and $D, D'$, and $A, A'$.

{\bf Case 4.} Let $K=P(p,q,r)$ be a knot such that $p,r>0$ and odd with $q<0$ and even. This case is identical to Case 3, except that the region of $\kw$ representing $q$ should be constructed in the clockwise direction. 

{\bf Case 5.} Let $K=P(p,q,r)$ be a knot such that $p<0$ and odd, $r>0$ and odd, and $q>0$ and even. This case is very close to Case 3, except that the region of $\kw$ representing $p$ should be constructed in the clockwise direction. In the first step,   since $p<0$, we observe the strand for $C$ lies over the strand for $A$. Thus when we fold $C$ over and up to rest along $B$, we change the square of ribbon to become a fold from angle $\nicefrac{\pi}{2}$. (This means we don't use any extra ribbonlength when identifying $C, C'$). The rest of the steps are identical to Case~3.
\qed
\end{const}

We are now ready to prove Theorem~\ref{thm:pretzel}: any  pretzel knot type  $K=P(p,q,r)$ contains a folded ribbon knot $\kw$ such that  $\Rib(\kw)\le2(|p|+|q|+|r|)+2.$

\begin{proof}[Proof of Theorem~\ref{thm:pretzel}] 

Set the width $w=1$, and recall from Remark~\ref{rmk:squares} that the ribbonlength of each region with $p$ half-twists is $2|p|$ plus the length of the ends. The folded ribbonlength of a pretzel knot from Construction~\ref{const:pretzel} is thus $2(|p|+|q|+|r|)$ plus the ribbonlength of any identifications between the twist regions. The arrangement of the twist regions meant that most of these identifications did not contribute to the folded ribbonlength. In Cases 1 and 2, there is an extra 2 units of folded ribbonlength from identifying the ends $A, A'$ and $D, D'$. In Cases 3 and 4, there is an extra 2 units of folded ribbonlength from identifying the ends $C, C'$ and $F, F'$. In Case 5, there is an extra unit of folded ribbonlength from identifying the ends $F, F'$
\end{proof}

\begin{remark}
By following the proof of Theorem~\ref{thm:pretzel}, we can deduce the number of sticks used in Construction~\ref{const:pretzel} of the pretzel knot $P(p,q,r)$. There are $2(|p|+|q|+|r|)$ sticks from the twist regions, plus additional sticks for connections. 
\begin{compactitem}
\item In Cases 1 and 2 we use $2(|p|+|q|+|r|+2)$ sticks, since there are an extra 2 sticks from connecting each of $A,A'$ and $D,D'$.
\item In Cases 3 and 4 we use $2(|p|+|q|+|r|+1) +1$ sticks, since there are an extra 2 sticks from identifying each of the ends $C, C'$ and $F, F'$, while identifying ends $E, E'$ reduces the number of sticks by 1.
\item In Case 5 we use $2(|p|+|q|+|r|+1)$ sticks, since 1 stick is added from identifying $C, C'$, 2 sticks are added from identifying $F,F'$, and 1 stick is removed from identifying $E, E'$.
\end{compactitem}
Thus, the resulting folded ribbon pretzel knot is a topological annulus in Cases 1, 2, and 5 and a topological M\"obius band in cases 3 and 4.
\end{remark}

We now add in our knowledge of crossing number to Theorem~\ref{thm:pretzel}

\begin{corollary}\label{cor:pretzel}
Let $K=P(p,q,r)$ be a pretzel knot type, and assume that either (a)  $p,q,r$ have the same sign, or (b) $K=P_{-p,q,r}$ and $p,q,r\ge2$. Then $K$ contains a folded ribbon knot $\kw$ with  
$$\Rib(\kw)\le 2\Cr(K)+2.$$
\end{corollary}
\begin{proof}
Apply Remark~\ref{rmk:pretzel} to Theorem~\ref{thm:pretzel}.
\end{proof}

\begin{corollary}\label{cor:twist}
Any twist knot type $K=T_n$ contains a folded ribbon knot $\kw$ with
$$\Rib(\kw)\le 2n+6=2\Cr(T_n)+2.$$
\end{corollary}
\begin{proof}
Without loss of generality, assume $n>0$, and recall that all twist knots $T_n$ are pretzel knots of the form $P(n,1,1)$. We then use Theorem~\ref{thm:pretzel}, 
to deduce
 $$\Rib(\kw)\le2(n+1+1)+2=2(n+2)+2=2(\Cr(T_n))+2.$$
\end{proof}

Corollaries~\ref{cor:pretzel} and~\ref{cor:twist} both give a linear upper bound on folded ribbonlength in terms of crossing number for twist knots and certain types of pretzel knots.  Construction~\ref{const:pretzel} is currently the only known model for folded ribbon pretzel knots, and thus gives a new approach to the study of folded ribbon twist knots. For a twist knot $K=T_n$, we can compare our upper bound on ribbonlength for folded ribbon twist knots, $\Rib(\kw)\leq 2n+6$, to the bounds given by Tian \cite{Tian}. She found that  
$$\Rib(\kw)\le \frac{\sqrt{5}+1}{2}n+5+\sqrt{5}+\sqrt{\frac{5+\sqrt{5}}{2}}.$$ 
Comparing coefficients, we see $\frac{\sqrt{5}+1}{2}<2$, and so we know that Tian's bound will be smaller than ours as $n$ gets large. Indeed, a computation in  {\em Mathematica} shows our ribbonlength is smaller than Tian's only when $n<8.216$. Of interest is the special case $T_2$, the figure-eight knot. Kaufman \cite{Kauf05} shows that the folded ribbonlength of the figure-eight knot is bounded above by $\frac{40}{\sqrt{15}}\leq 10.328$, while Corollary~\ref{cor:twist} lowers this bound down to 10.

\section{Ribbonlength of $(p,q)$ torus knots}\label{sect:pq-torus}

In this section we give a new construction for a folded ribbon $(p,q)$ torus knot and use this to find its folded ribbonlength. It turns out that this construction gives a sub-linear upper bound on folded ribbonlength in terms of crossing number for torus knots with $p\geq q>2$, and a linear upper bound for $(p,2)$ torus knots.

Previous constructions of $(p,q)$ torus knots \cite{DeM, KMRT, Tian} used in folded ribbonlength computations have all assumed that $p\leq q$; that the number of wraps around the meridian is greater than the number of turns around the longitude. In this section, we will assume the reverse: $p\geq q$. In other words, we will assume that the number of turns around the longitude is greater than the number of wraps around the meridian. Topologically, there is no difference between these cases, since we know that the $(p,q)$ torus knot is equivalent to the $(q,p)$ torus knot~\cite{Adams}. However, we shall see that the different geometry of these torus knots gives different bounds on folded ribbonlength.

\begin{const}[Folded ribbon $T(p,q)$ links]\label{const:pq-torus}
Without loss of generality, we assume $p\ge q \ge 2$, and we think of the $T(p,q)$ link as having $p$ strands (around the longitude) that are wrapped $q$ times around the meridian. In Figure~\ref{fig:torus-step1} (left), we see a truncated diagram representing the trefoil knot\footnote{In some texts (like \cite{JohnHen}), Figure~\ref{fig:torus-step1} (left) represents $T(3,-2)$. The construction for $T(3,2)$ is very similar.} $T(3,2)$.  We understand this representation by observing three strands labeled 1, 2, 3 correspond to the three turns around the longitude. The two wraps around the meridian are shown when the first strand is wrapped behind the other two, then the second strand is wrapped behind the other two. Finally, the appropriate ends are identified: $1$ to $1'$, and $2$ to $2'$, and $3$ to $3'$.

\begin{figure}[htbp]
    \centering
    \begin{overpic}{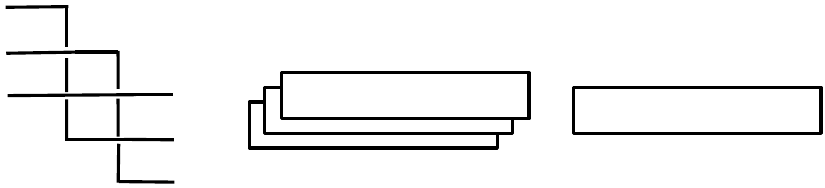}
    \put(-1,21){$1$}
    \put(-1,15.5){$2$}
    \put(-1,10){$3$}
    \put(21.5,10.5){$1'$}
    \put(21.5,5){$2'$}
    \put(21.5,0){$3'$}
    \put(34,15){$1$}
    \put(32,13){$2$}
    \put(30,11){$3$}
    \put(67.5,11.25){$1$}
    \put(67.5,8.5){$2$}
    \put(67.5,6){$3$}
    \end{overpic}
    \caption{On the left, a truncated $T(3,2)$ knot. In the middle and right, the labeling of the left end of a stack of three strands of ribbon of width $w$.}
    \label{fig:torus-step1}
\end{figure}
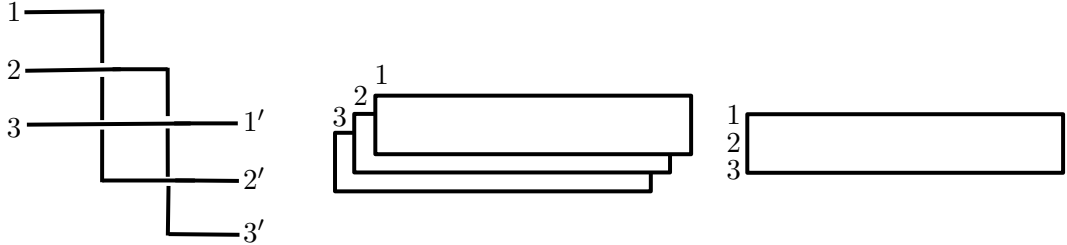

Step 1. With this set up in mind, take a ribbon of width $w$, and lay $p$ horizontal strands of ribbon directly atop one another, as illustrated in the middle and right of Figure~\ref{fig:torus-step1}. Label the left end of each strand such that the topmost end reads $1$, the end underneath the topmost end reads $2$, etc. until the bottom end reads $p$. Figure~\ref{fig:torus-step1} (middle and right) shows this for the $T(3,2)$ knot.

\begin{figure}[htbp]
    \centering
    \begin{overpic}{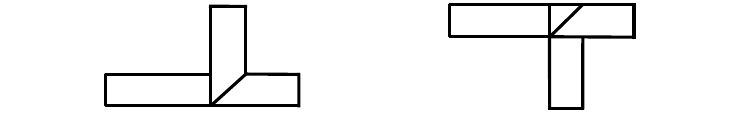}
    \put (11.5,5) {1}
    \put (11.5,2.25) {2}
    \put (11.5,-0.5) {3}
    \put (57.5,14.5) {1}
    \put (57.5, 11.75) {2}
    \put (57.5,9) {3}
    \put (85.5,11.5) {$1'$}
    \put (73, -2.25) {$2'$}
    \put (76, -2.25) {$3'$}
    \end{overpic}
    \caption{Folding $q$ right ends up (left) and then down behind the stack (right).}
    \label{fig:torus-step2}
\end{figure}
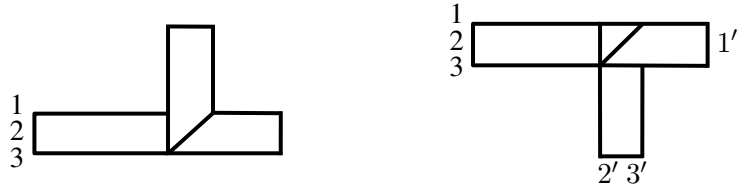

Step 2.  Here, we wrap $q$ of the strands around the meridian. Take the right end of the topmost strand and fold it up 90 degrees so that it rests perpendicular to the stack of strands (Figure~\ref{fig:torus-step2}, left), then fold it down under the stack (Figure~\ref{fig:torus-step2}, right). Repeat this process with the right ends of strands $2,3,\dots,q$ so that $q$ ends have been folded down under the stack. There are now two sets of unlabeled ends: those to the right and those folded under the stack. (If $p=q$, every end will have been folded under the stack.) We first label the set of ends on the right such that the topmost end reads $1'$, the end underneath the topmost end reads $2'$, etc. until the bottom end reads $(p-q)'$. We next label the ends which have been folded under the stack such that the topmost end reads $(p-q+1)'$, etc. until the bottom end reads $p'$. In this second case, we label the strands from left to right as we go from top to bottom. Note that this labeling precisely mirrors the labeling in the diagram on the left of Figure~\ref{fig:torus-step1}, and is shown specifically for the $T(3,2)$ knot on the right of Figure~\ref{fig:torus-step2}.

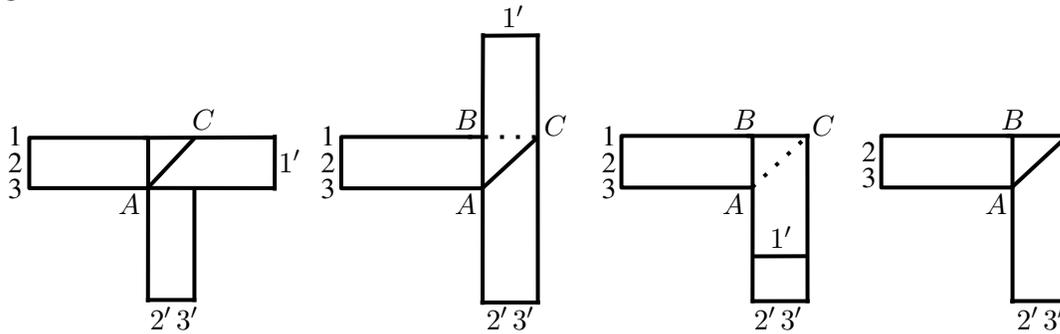
\begin{figure}[htbp]
    \centering
    \begin{overpic}{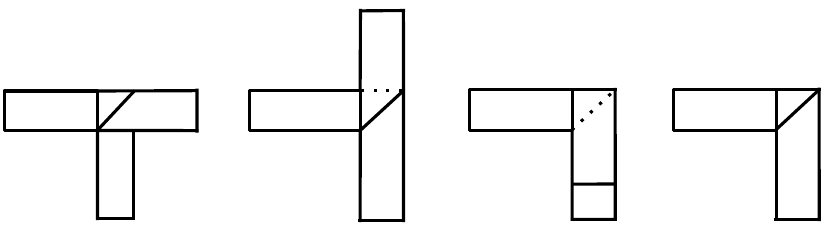}
    \put (-1.5,15.5) {1}
    \put (-1.5,13) {2}
    \put (-1.5,10.5) {3}
    \put (24.25,13) {$1'$}
    \put (12,-2) {$2'$}
    \put (14.5,-2) {$3'$}
    \put(9,9){$A$}
    \put(16,17){$C$}
    \put (28.25, 15.5) {1}
    \put (28.25, 13) {2}
    \put (28.25, 10.5) {3}
    \put (45.5, 26.75) {$1'$}
    \put (44, -2) {$2'$}
    \put (46.5,-2) {$3'$}
     \put(41,9){$A$}
    \put(41,17){$B$}
    \put(49.5,16.5){$C$}
    \put(55,15.5) {1}
    \put (55, 13) {2}
    \put (55, 10.5) {3}
    \put(71, 5.75){$1'$}
    \put (69.5, -2) {$2'$}
    \put (72,-2) {$3'$}
    \put(66.5,9){$A$}
    \put(67.5,17){$B$}
    \put(75,16.5){$C$}
    \put (79.75, 14) {2}
    \put (79.75, 11.5) {3}
    \put (94.5, -2) {$2'$}
    \put (97,-2) {$3'$}
    \put(91.5,9){$A$}
    \put(93,17){$B$}
  \end{overpic}
    \caption{Folding the right $1'$ end $90^\circ$ up over the center diagonal fold line $AC$ (left and second from left). Next, folding $1'$ down along $BC$ and then folding  $90^\circ$ across $AC$ to end up at $1$ to the left (second from right). Finally, joining $1$ to $1'$ and shrinking the connections (far right).}
    \label{fig:torus-step3}
\end{figure}

Step 3. We identify the ends labeled $1$ and $1'$ as follows. Referring to Figure~\ref{fig:torus-step3} left and second from left, we first fold the piece of ribbon ending at $1'$ up  along the diagonal $AC$ creating a fold from angle $\nicefrac{\pi}{2}$.  Second, we fold $1'$ down along the horizontal line $BC$ so that $1'$ rests along the piece of ribbon ending at $2'$ (second from right in Figure~\ref{fig:torus-step3}). Third, we again fold $1'$ across the diagonal $AC$ so it rests along the piece of ribbon ending at $1$. This last action creates another fold from angle $\nicefrac{\pi}{2}$. Finally, we can join the pieces of ribbon at $1$ and $1'$ creating a fold from angle $0$, then shrink the length of this connection so that the fold line lies along $AB$ (far right in Figure~\ref{fig:torus-step3}).

We repeat this process, identifying ends $2$ to $2'$, then $3$ to $3'$, $\dots$, then $(p-q)$ to $(p-q)'$. There are now no unconnected ends on the right. This point in the construction is illustrated on the right in Figure~\ref{fig:torus-step3} for the $T(3,2)$ knot.

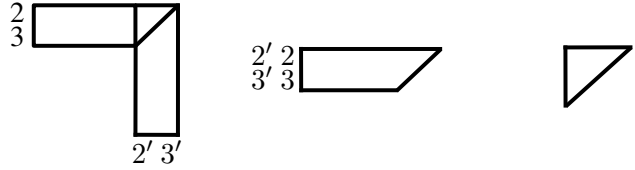
\begin{figure}[htbp]
    \centering
    \begin{overpic}{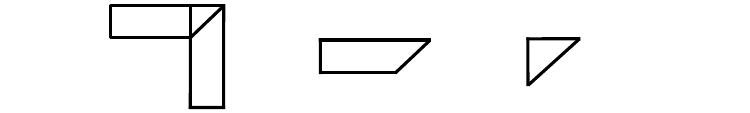}
    \put (12.25, 12.5) {2}
    \put (12.25, 10) {3}
    \put (25,-2.25) {$2'$}
    \put (28,-2.25) {$3'$}
    \put (37.5, 8) {$2'$ $2$}
    \put (37.5, 5.5) {$3'$ $3$}
    %
    \end{overpic}
    \caption{Folding $2'$, then $3'$ up across the diagonal center fold (middle). Then identifying the ends and shrinking the connections (right).}
    \label{fig:torus-step4}
\end{figure}

Step 4. We next identify the ends labeled $(p-q+1)$ and $(p-q+1)'$ by folding the piece of ribbon ending at $(p-q+1)'$ up and along the diagonal center fold so that it rests along the piece of ribbon ending at $(p-q+1)$. This creates a fold from angle $\nicefrac{\pi}{2}$, and is shown in the middle image of Figure~\ref{fig:torus-step4}. We join the pieces of ribbon at $(p-q+1)$ and $(p-q+1)'$ creating a fold from angle $0$. We then shrink the length of this connection so that the fold line lies along the left edge of the center region. Repeat this process for all pairs of corresponding ends through $p$ and $p'$. The final construction for the $T(3,2)$ knot is shown on the right in Figure~\ref{fig:torus-step4}.   \qed
\end{const}

We next calculate the ribbonlength of Construction~\ref{const:pq-torus}. To do this, set the width $w=1$ and observe from Remark~\ref{rmk:squares} that the ribbonlength of a unit square of ribbon is 1~unit. We also recall from Corollary~\ref{cor:triangles} that the ribbonlength of a fold from angle $\nicefrac{\pi}{2}$ is $1$~unit.  We are now ready to prove Theorem~\ref{thm:pq-torus}: any $(p,q)$ torus link type $L$ (with $p\geq q\ge2$) contains a folded ribbon link $\lw$ such that $\Rib(\lw)\le 2p$.

\begin{proof}[Proof of Theorem~\ref{thm:pq-torus}]
Since $T(p,q)\cong T(q,p)$, we can use Construction~\ref{const:pq-torus} for all torus links. We also assume the ribbon width $w=1$. In Step 1, we begin with a stack of $p$ horizontal strands of ribbon. We break this stack into two smaller stacks with $q$ and $p-q$ strands of ribbon, respectively; and assume the stack of $q$ strands is on top of the stack of $p-q$ strands. We follow Step 2 to calculate the ribbonlength created by folding the $q$ strands underneath the stack of $p-q$ strands. In order to fold a single right end up so that it rests perpendicular to its left end, we form a fold from angle $\nicefrac{\pi}{2}$ which has ribbonlength 1. When we fold this strand down beneath the other strands, we form one unit square of ribbon. For one strand, then, this move creates $1+1=2$ units of folded ribbonlength. We repeat this process for the entire stack of $q$ strands; thus we create $2q$ units of folded ribbonlength.

Next, we will calculate the ribbonlength created by connecting the corresponding pairs of ends together. We begin with Step 3, where we join the ends $1$ to $1'$.  In this process we form two folds from angle $\nicefrac{\pi}{2}$ with ribbonlength $1+1=2$ units. We repeat this process $p-q$ times and so get an additional $2p-2q$ units of folded ribbonlength.

Finally, we move to Step 4 and join the remaining $q$ ends by folding each one across and up across the center diagonal fold. After shrinking the joined ribbon, we see this move replaces the unit square of ribbon (already counted) on the back with a fold from angle $\nicefrac{\pi}{2}$. Since these have the same ribbonlength, we see the total ribbonlength is not changed.  After adding all the components together, we obtain $\Rib(\lw)\le 2q+2p-2q=2p$.
\end{proof}

\begin{remark}
We see that Construction~\ref{const:pq-torus} of a $(p,q)$ torus knot uses $3q+4(p-q)+q = 4p$ sticks. The resulting folded ribbon torus knot is always a topological annulus.
\end{remark}

\subsection{Comparison and sub-linear growth}

We now compare Theorem~\ref{thm:pq-torus} to the many results about $(p,q)$ torus knots that have gone before.  Firstly, when we apply Theorem~\ref{thm:pq-torus} to the trefoil knot, $T(3,2)$, we see the trefoil knot has a folded ribbonlength of 6, which is lower than the previous bounds (\cite{Kauf05, KMRT}).  

Secondly,  we restrict our attention just to $T(p,2)$ torus knots, which we recall are equivalent to $T(2,p)$ torus knots. Constructions~\ref{const:2p-torus} and~\ref{const:pq-torus} give two different ways of creating these folded ribbon knots, but both constructions show that in each $T(p,2)$ knot type, there is a folded ribbon knot $\kw$ with $\Rib(\kw)\leq 2p =2\Cr(T(p,2))$. Thus Theorem~\ref{thm:pq-torus} gives a second proof of Theorem~\ref{thm:2p-torus}, and we again have a bound on folded ribbonlength that is linear in crossing number. We note Construction~\ref{const:2p-torus} uses $2p+2$ sticks, while Construction~\ref{const:pq-torus} uses $4p$ sticks.

Finally, we use {\em Mathematica} to compare our bound from Theorem~\ref{thm:pq-torus} to all other known bounds for torus knots. It turns out that for $p,q\geq2$, our bound is always lower. 
\begin{itemize}
\item Tian \cite{Tian}, $T(p,q)$ knots: $2p < 4pq$.
\item Kennedy {\em et al.} \cite{KMRT} looked at specific families of torus knots.
\begin{itemize}
\item $T(q+1,q)$ knots (including $T(3,2)$): $2q+2<(2q+1)\cot(\frac{\pi}{2q+1})$.
\item $T(2q+1,q)$ knots (including $T(5,2)$): $4q+2<(2q+1)\cot(\frac{\pi}{2(2q+1)})$.
\item $T(2q+2,q)$ knots: $4q+4<(2q+2)\cot(\frac{\pi}{2q+2})$.
\item $T(2q+4,q)$ knots: $4q+8<(2q+4)\cot(\frac{\pi}{2q+4})$.
\item $T(p,2)$ knots for $p\geq 7$, odd: $2p<p\cot(\frac{\pi}{p})$.
\end{itemize}
\end{itemize}

There is a simple reason why the ribbonlength bound in Theorem~\ref{thm:pq-torus} is lower than any of the others. In Theorem~\ref{thm:torus-crossing} below, we show that our construction yields an upper bound for ribbonlength that is sub-linear in crossing number. The other previously known bounds for ribbonlength are either linear or quadratic in crossing number.

We are ready to prove Theorem~\ref{thm:torus-crossing}: Suppose $L$ is an infinite family of ($p,q$) torus link types, where $p, q> 2$ and $p=aq+b$ for some $a,b\in\Z_{\ge 0}$. Then for each $q=3,4,5,\dots$, the family $L$ contains a folded ribbon link $\lw$ with  
 $$ \Rib(\lw)\le \sqrt{6}(a+\frac{b}{3})(\Cr(L))^\frac{1}{2}.$$

\begin{proof}[Proof of Theorem~\ref{thm:torus-crossing}]
Given any $p\ge q> 2$, we  aim to find a $c>0$ such that $\Rib(\lw)\leq c\cdot \Cr(L)^\frac{1}{2}$.  We start by squaring both sides $(\Rib(\lw))^2\le (c)^2 \Cr(L)$. Next, we divide both sides by the crossing number, and use Theorem~\ref{thm:pq-torus}, resulting in 
$$\frac{(\Rib(\lw))^2}{\Cr(L)}\le\frac{(2p)^2}{p(q-1)}\le (c)^2.$$
 Since $q\le p$, we see that $q^2-q\le pq-p$ and $\frac{1}{q^2-q}\ge\frac{1}{pq-p}$. Hence, 
$$
    \frac{4p^2}{p(q-1)} \le \frac{4p^2}{q^2-q} \le \frac{4p^2}{q^2(1-\frac{1}{q})}.
    $$
Since we set $q\ge3$, we see that $\frac{1}{q}\le\frac{1}{3}$. It follows that $1-\frac{1}{q}\ge\frac{2}{3}$, so $\frac{1}{1-\frac{1}{q}}\le\frac{1}{\frac{2}{3}}=\frac{3}{2}$. Thus, we find that 
$$\frac{(\Rib(\lw))^2}{\Cr(L)} \le\frac{4p^2}{q^2(1-\frac{1}{q})}\le 6 \frac{p^2}{q^2}=(c)^2,$$ 
and subsequently that 
$$
    (\Rib(\lw))^2  \le 6\frac{p^2}{q^2}\Cr(L) \quad \text{or} \quad  \Rib(\lw) \le \sqrt{6}\frac{p}{q}(\Cr(L))^\frac{1}{2}.
$$
Recall that the parameters $p$ and $q$ are integers and $p\ge q$, so the Remainder Theorem allows us to write $p=aq+b$ for some $a,b\in\Z_{\ge 0}$. We can manipulate this equation to read $\frac{p}{q}=a+\frac{b}{q}$. Since $\frac{1}{q}\le\frac{1}{3}$, we have that $\frac{p}{q}\le a+\frac{b}{3}$. Altogether, we see that 
   $$ \Rib(\lw)\le \sqrt{6} (a+\frac{b}{3})(\Cr(L))^\frac{1}{2}.
$$
\end{proof}

Since  Theorem~\ref{thm:pq-torus} applies to every $(p,q)$ torus link where $p,q>2$, the proof of Theorem~\ref{thm:torus-crossing} shows the following.
\begin{corollary}
For all $p,q> 2$ there is  a constant $c>0$ such that a $(p,q)$ torus link type $L$ contains a folded ribbon link $\lw$ with $\Rib(\lw) \le c\cdot (\Cr(L))^\frac{1}{2}$. 
\end{corollary}

\begin{proof} From Theorem~\ref{thm:torus-crossing}, we know that if $p=aq+b$ for $a,b\in\Z_{\ge 0}$, then $c= \sqrt{6}(a+\frac{b}{3})$.
\end{proof}

We end with a simple example, where we apply Theorem~\ref{thm:torus-crossing}  to the infinite family of torus knot types $T(2q+1,q)$. For this family we have $a=2$ and $b=1$ in Theorem~\ref{thm:torus-crossing}. Hence for each $q=3,4,5,\dots$, the $(2q+1,q)$ torus knot type $K$ contains a folded ribbon knot $\kw$ with $$\Rib(\kw)\le \frac{7}{3}\sqrt{6}(\Cr(K))^\frac{1}{2}.$$


\section{Conclusion} \label{sect:conclusion}

In this paper, we have provided four examples of infinite knot families where the upper bound on the folded ribbonlength is linear in crossing number. All of the constructions are new. (Theorem~\ref{thm:2-bridge} for 2-bridge knots, Construction~\ref{const:2p-torus} for $(2,q)$ torus knots, Construction~\ref{const:pretzel} for twist and certain pretzel knots.)  We  also found a new construction for any $(p,q)$ torus link (Construction~\ref{const:pq-torus}) that gave an upper bound on folded ribbonlength that is sub-linear in crossing number when $p\geq q>2$.

Along the way, we computed the number of edges or sticks in our polygonal knot diagrams. When we compared our stick numbers with the stick numbers from previous work (\cite{Kauf05, KMRT, Tian}), we noticed that we used many more sticks in our constructions. The extra flexibility given by having more sticks has allowed us to reduce the folded ribbonlength in most cases.  A striking example of this is the trefoil knot. Kaufman \cite{Kauf05} used 5 sticks with folded ribbonlength bounded by $5\cot(\nicefrac{\pi}{5})\le 6.882$. In contrast, we used 8 sticks in Construction~\ref{const:2p-torus} and 12 sticks in Construction~\ref{const:pq-torus}, and lowered the folded ribbonlength to 6 in both constructions.  As another interesting example, consider $(p,q)$ torus knots and links where we assume $p\geq q\ge 2$. The five torus knot families explored by Kennedy {\em et al.} \cite{KMRT} all use $p$ sticks (the conjectured minimum), while our Construction~\ref{const:pq-torus} uses $4p$ sticks. The difference allowed us to have an upper bound in folded ribbonlength that is sub-linear in crossing number (when $p\geq q>2$). Having seen these examples, the reader might expect that increasing the number of sticks always means the folded ribbonlength decreases. Interestingly, this is not the case! Tian \cite{Tian} constructed a figure-eight knot with 12 sticks and ribbonlength bounded above by $12.375$, Kauffman used 6 sticks and ribbonlength bounded above by $\nicefrac{40}{\sqrt{15}}\le 10.328$, while we used 10 sticks and lowered the ribbonlength to 10 units.  The relationship between the number of sticks in a polygonal knot diagram and the corresponding folded ribbonlength is more subtle than expected.

There are two comments worth making here. First, Colin Adams {\em et al.} have looked at the {\em planar stick index} (see \cite{Adams-PS}) and {\em projection stick index} (see \cite{Adams-Shayler}) of a knot or link. 
The planar stick index $\pl$ is the  minimum number of sticks needed for a polygonal diagram of $K$. The projection stick index $\prs$ is the least number of nontrivial projection sticks from a projection of a polygonal knot. In other words, each stick that does not project to a point counts as a stick for $\prs$. To understand the difference in the definitions, consider a knot embedded in the cubic lattice (as described in Section~\ref{sect:2-bridge}). Then imagine projecting a staircase in a vertical plane to the $xy$-plane\footnote{With thanks to Colin Adams for explaining this example to the first author.}. This would count as 1 stick for $\pl$, but as many sticks for $\prs$.   These two definitions appear to be closely related, and yet we expect them to be distinct. However, there is no specific example showing this fact. (Another open question.) Having said that, it appears that the planar stick index is the most relevant to folded ribbon knots. All of our examples indicate that within a knot type, polygonal knots with more than the planar stick number of edges tend to have a lower folded ribbonlength.  There is much to consider here, and understanding the subtle relationship between folded ribbonlength and planar stick number is a natural question to pursue.

Secondly, as mentioned in Section~\ref{sect:ribbonlength}, we compute our upper bounds on folded ribbonlength with respect to knot diagram equivalence, not folded ribbon equivalence. Our examples show that we expect there to be a  difference between folded ribbonlength bounds between the two types of equivalence. For example, Kaufman's trefoil knot is a topological M\"obius band, while ours is a topological annulus. In previous work \cite{DKTZ}, we showed that the ribbon linking number was different for two $(5,2)$ torus knots (given by Kennedy {\em et al.} \cite{KMRT}) which were constructed with different numbers of sticks. There is much to explore here. 

Finally, the ribbonlength crossing number problem is still open. Can we find upper and lower bounds on folded ribbonlength for more families of knots? Can we prove that for all knots and links, the folded ribbonlength is bounded above linearly in crossing number? Alternatively, can we find a family of knots or links where the lower bound for folded ribbonlength has super-linear growth in crossing number?


\section{Acknowledgments}
We give thanks to the referee for their helpful comments which have greatly improved the quality of the paper.
In addition, we wish to thank Jason Cantarella for his insights and acting as a sounding board for many of our proofs. The first author also  thanks Colin Adams for clarifying the difference between the planar stick index and projection stick index. 

Denne's research has been funded over many summers by Lenfest Grants from Washington \& Lee University, most recently in 2018, 2019, and 2020. 
 The research completed by Meehan in Fall 2011 was funded by the Center for Women in Mathematics at Smith College (NSF grant DMS 0611020). The research completed by Haden and Larsen in Summer 2020 was funded by Washington \& Lee's Summer Research Scholars program. 
 
The work on folded ribbon knots has been developed with Denne's undergraduate students over many years. We all wish to thank the students who have contributed to our understanding of folded ribbon knots.
\begin{compactitem}
\item Shivani Aryal and Shorena Kalandarishvili: funded by Smith College's 2009 Summer Undergraduate Research Fellowship program.
\item Eleanor Conley and Rebecca Terry: funded by the Center for Women in Mathematics at Smith College Fall 2011, which was funded by NSF grant DMS 0611020. 
\item Mary Kamp and Catherine (Xichen) Zhu: funded by 2015 Washington \& Lee Summer Research Scholars program.
\item Corinne Joireman and Allison Young: funded by 2018 W\&L Summer Research Scholars program.
\end{compactitem}


\bibliography{folded-ribbons}{}

\begin{thebibliography}{10}

\bibitem{Adams-PS}
Colin Adams, Dan Collins, Katherine Hawkins, Charmaine Sia, Rob Silversmith,
  and Bena Tshishiku.
\newblock Planar and spherical stick indices of knots.
\newblock {\em J. Knot Theory Ramifications}, 20(5):721--739, 2011.

\bibitem{Adams-Shayler}
Colin Adams and Todd Shayler.
\newblock The projection stick index of knots.
\newblock {\em J. Knot Theory Ramifications}, 18(7):889--899, 2009.

\bibitem{Adams}
Colin~C. Adams.
\newblock {\em The knot book}.
\newblock W. H. Freeman and Company, New York, 1994.
\newblock An elementary introduction to the mathematical theory of knots.

\bibitem{Ashley}
Clifford~W. Ashley.
\newblock {\em The {A}shley {B}ook of {K}nots}.
\newblock Doubleday, Garden City, New York, 1944.

\bibitem{BS99}
Gregory Buck and Jonathan Simon.
\newblock Thickness and crossing number of knots.
\newblock {\em Topology Appl.}, 91(3):245--257, 1999.

\bibitem{BZ}
Gerhard Burde and Heiner Zieschang.
\newblock {\em Knots}, volume~5 of {\em De Gruyter Studies in Mathematics}.
\newblock Walter de Gruyter \& Co., Berlin, second edition, 2003.

\bibitem{CKS}
Jason Cantarella, Robert~B. Kusner, and John~M. Sullivan.
\newblock On the minimum ropelength of knots and links.
\newblock {\em Invent. Math.}, 150(2):257--286, 2002.

\bibitem{Crom}
Peter~R. Cromwell.
\newblock {\em Knots and links}.
\newblock Cambridge University Press, Cambridge, 2004.

\bibitem{DeM}
L.~DeMaranville.
\newblock Construction of polygons by tying knots with ribbons.
\newblock Master's thesis, CSU Chico, 1999.

\bibitem{Den-FRS}
Elizabeth Denne.
\newblock Folded ribbon knots in the plane.
\newblock In Colin Adams, Erica Flapan, Allison Henrich, Louis Kauffman, Lew
  Ludwig, and Sam Nelson, editors, {\em A Concise Encyclopedia of Knot Theory},
  chapter~88, pages 877--897. Chapman and Hall/CRC, Boca Raton, FL, 2021.

\bibitem{Den-FRC}
Elizabeth Denne.
\newblock Ribbonlength and crossing number for folded ribbon knots.
\newblock {\em J. Knot Theory Ramifications}, 30(4):2150028, 21, 2021.

\bibitem{DKTZ}
Elizabeth Denne, Mary Kamp, Rebecca Terry, and Xichen Zhu.
\newblock Ribbonlength of folded ribbon unknots in the plane.
\newblock In {\em Knots, links, spatial graphs, and algebraic invariants},
  volume 689 of {\em Contemp. Math.}, pages 37--51. Amer. Math. Soc.,
  Providence, RI, 2017.

\bibitem{Den-FRLU}
Elizabeth Denne and Troy Larsen.
\newblock Linking number, unknots and folded ribbon knots.
\newblock Preprint in preparation, 2021.

\bibitem{DSW}
Elizabeth Denne, John~M. Sullivan, and Nancy Wrinkle.
\newblock Ribbonlength for knot diagrams.
\newblock Preprint in preparation.

\bibitem{DEY}
Yuanan Diao, Claus Ernst, and Xingxing Yu.
\newblock Hamiltonian knot projections and lengths of thick knots.
\newblock {\em Topology Appl.}, 136(1-3):7--36, 2004.

\bibitem{gm}
Yuanan Diao, Claus Ernst, and Xingxing Yu.
\newblock Hamiltonian knot projections and lengths of thick knots.
\newblock {\em Topology Appl.}, 136(1-3):7--36, 2004.

\bibitem{DEZ}
Yuanan Diao, Claus Ernst, and Uta Ziegler.
\newblock The linearity of the ropelengths of {C}onway algebraic knots in terms
  of their crossing numbers.
\newblock {\em Kobe J. Math.}, 28(1-2):1--19, 2011.

\bibitem{DK}
Yuanan Diao and Robert Kusner.
\newblock Private communication, July 2012.

\bibitem{HHKNO}
Youngsik Huh, Kyungpyo Hong, Hyoungjun Kim, Sungjong No, and Seungsang Oh.
\newblock Minimum lattice length and ropelength of 2-bridge knots and links.
\newblock {\em J. Math. Phys.}, 55(11):113503, 11, 2014.

\bibitem{John}
D.A. Johnson.
\newblock {\em Paper {F}olding for the {M}athematics {C}lass}.
\newblock Washington D.C. National Council of Teachers of Mathematics, 1957.

\bibitem{JohnHen}
Inga Johnson and Allison Henrich.
\newblock {\em An {I}nteractive {I}ntroduction to {K}not {T}heory}.
\newblock Dover Publications, 2017.

\bibitem{Kauf87}
Louis~H. Kauffman.
\newblock State models and the {J}ones polynomial.
\newblock {\em Topology}, 26(3):395--407, 1987.

\bibitem{Kauf05}
Louis~H. Kauffman.
\newblock Minimal flat knotted ribbons.
\newblock In {\em Physical and numerical models in knot theory}, volume~36 of
  {\em Ser. Knots Everything}, pages 495--506. World Sci. Publ., Singapore,
  2005.

\bibitem{KMRT}
Brooke Kennedy, Thomas~W. Mattman, Roberto Raya, and Dan Tating.
\newblock Ribbonlength of torus knots.
\newblock {\em J. Knot Theory Ramifications}, 17(1):13--23, 2008.

\bibitem{LeeJin}
Hwa~Jeong Lee and Gyo~Taek Jin.
\newblock Arc index of pretzel knots of type {$(-p, q, r)$}.
\newblock {\em Proc. Amer. Math. Soc. Ser. B}, 1:135--147, 2014.

\bibitem{lsdr}
R.~A. Litherland, J.~Simon, O.~Durumeric, and E.~Rawdon.
\newblock Thickness of knots.
\newblock {\em Topology Appl.}, 91(3):233--244, 1999.

\bibitem{Liv}
Charles Livingston.
\newblock {\em Knot theory}, volume~24 of {\em Carus Mathematical Monographs}.
\newblock Mathematical Association of America, Washington, DC, 1993.

\bibitem{McC}
Cynthia~L. McCabe.
\newblock An upper bound on edge numbers of {$2$}-bridge knots and links.
\newblock {\em J. Knot Theory Ramifications}, 7(6):797--805, 1998.

\bibitem{Mur87}
Kunio Murasugi.
\newblock Jones polynomials and classical conjectures in knot theory.
\newblock {\em Topology}, 26(2):187--194, 1987.

\bibitem{Schu}
Horst Schubert.
\newblock \"{U}ber eine numerische {K}noteninvariante.
\newblock {\em Math. Z.}, 61:245--288, 1954.

\bibitem{Thi}
Morwen~B. Thistlethwaite.
\newblock A spanning tree expansion of the {J}ones polynomial.
\newblock {\em Topology}, 26(3):297--309, 1987.

\bibitem{Tian}
Grace~M. Tian.
\newblock Linear upper bound on the ribbonlength of torus knots and twist
  knots.
\newblock Summer Research Report,
  \url{https://math.mit.edu/research/highschool/rsi/documents/2017Tian.pdf}
  (accessed August 6, 2020), June 2017.

\bibitem{Wel}
David Wells.
\newblock {\em The {P}enguin dictionary of curious and interesting geometry}.
\newblock Penguin Books, New York, 1991.

\end{thebibliography}
\bibliographystyle{plain}


\end{document}